\documentclass[twoside,12pt]{amsart}
\usepackage[all]{xy}
\usepackage{amssymb,amscd,longtable,amsmath}
\usepackage{booktabs}
\usepackage[dvips]{pict2e}
\theoremstyle{definition}
\newtheorem{thm}{Theorem}[section]
\newtheorem{prop}[thm]{Proposition}

\newtheorem{lem}[thm]{Lemma}
\newtheorem{rem}[thm]{Remark}
\newtheorem{cor}[thm]{Corollary}
\newtheorem{ex}[thm]{Example}
\newtheorem*{ack}{Acknowledgment}
\newtheorem*{mt}{Main Theorem}
\newcommand{\bracket}[1]{{\langle {#1} \rangle}}
\newcommand{\erase}[1]{}

\title{$K3$ surfaces and log del Pezzo surfaces of index three}
\author[H.~Ohashi]{Hisanori Ohashi}
\author[S.~Taki]{Shingo Taki}
\address{Graduate School of Mathematics,
 Nagoya University, Nagoya 464-8602, Japan}
\email{pioggia@kurims.kyoto-u.ac.jp}
\address{School of Information Environment, Tokyo Denki University,
2-1200 Muzai Gakuendai, Inzai-shi, Chiba 270-1382, Japan}
\email{taki@sie.dendai.ac.jp}
\date{\today}
\thanks{}
\begin{document}
\bibliographystyle{amsalpha+}

\begin{abstract}
We use classification of non-symplectic automorphisms of $K3$ surfaces to obtain a 
partial classification of log del Pezzo surfaces of index three. We characterize them as 
those with 
"Multiple Smooth Divisor Property", whose definition we will give.
Our methods include the construction of right resolutions of quotient singularities of index three
and an analysis of automorphism-stable elliptic fibrations on $K3$ surfaces combined with 
lattice theory.
In particular we find several log del Pezzo surfaces of Picard number one 
with non-toric singularities of index three. 
As a byproduct, we also obtain the holomorphic description of {\em{all}} 
non-symplectic automorphisms of $K3$ surfaces of order three 
whose fixed locus contains a smooth curve of genus $g\geq 2$ (called \em{of elliptic type}). 
\end{abstract}

\maketitle

\section{Introduction}\label{Introduction}

We work over the complex numbers $\mathbb{C}$. A normal complete surface $Z$ is called a 
{\em{log del Pezzo surface}} if 
it has only log terminal singularities and the anticanonical divisor $-K_Z$ is ample.
Log del Pezzo surfaces constitute one of the most interesting classes of rational surfaces;
they naturally appear in the outputs of the (log) minimal model program and their classification 
is an interesting problem. 
The {\em{index}} $k$ of $Z$ is the least positive integer such that $k K_Z$ is a Cartier divisor.

Log del Pezzo surfaces with index $k=1$ are sometimes called weak log del Pezzo surfaces and
their classification is a classical topic. 
In the index $k=2$ Alexeev and Nikulin \cite{AN}
(over $\mathbb{C}$)
and Nakayama \cite{Nakayama} (char. $p\geq 0$ and also for log pairs) 
gave complete classifications, whose methods are
independent in nature. 
For index $k\geq 3$, although there exist a large number of results, 
the complete classification is not available so far. 
In this direction we mention \cite{Dais} where toric 
log del Pezzo surfaces of index three with Picard number one are classified. 
We examine his list from our viewpoint in Section \ref{examples}.

The idea used in \cite{AN} was to relate log del Pezzo surfaces
to $K3$ surfaces. It might be compared to the researches of 
log Enriques surfaces, namely those surfaces $Z$ with only log terminal singularities,
$H^1(Z, \mathcal{O}_Z)=0$ and some multiple of the canonical divisor is trivial. 
They are often studied via $K3$ surfaces, too. In this direction we mention 
\cite{Zhang}, where Zhang described some log Enriques surfaces that arise as quotients 
of $K3$ surfaces by non-symplectic automorphisms of order three.
He uses automorphisms without fixed curves of genus $\geq 2$, while our focus in this paper is on  
automorphisms with those fixed curves. 
It might be interesting to note that non-symplectic automorphisms behave better
under our assumption, since the large genus curve
rigidifies the picture, for example as in Propositions \ref{ellip1} and \ref{howto}.

In this paper we want to discuss a possible generalization of 
the ideas of \cite{AN} to treat log del Pezzo surfaces of index three 
Accordingly we will give a
partial classification of log del Pezzo surfaces of index three and 
find some examples of 
non-toric log del Pezzo surfaces with Picard number one, see
Remark \ref{nontoric}.
Although our classification can cover only a restricted portion of log del Pezzo 
surfaces of index three, we remark that our approach has the following advantages:
since we start with $K3$ surfaces satisfying an appropriate condition, 
the characterizing property of our log del Pezzo surfaces becomes geometric. Thus our 
classification covers some surfaces with higher Picard number, compared to 
many previous results which assumed the Picard number to be one.
Also this implies that our description should be better suited for considerations of
family and moduli, since the moduli theory of $K3$ surfaces is well-understood.
On the other hand, our method includes the detailed descriptions of {\em{all}} 
non-symplectic automorphisms of $K3$ surfaces of order three 
whose fixed locus contains a smooth curve of genus $g\geq 2$ (called of elliptic type), 
compared to the descriptions of generic ones in \cite{Taki, AS}.

We can give a rough sketch of \cite{AN} as follows.
A log del Pezzo surface $Z$ of index $k\leq 2$ always has a smooth element $C\in |-2K_Z|$ disjoint
from singularities. This fact is called the {\em{smooth divisor theorem}}.
Also they define the "right" resolution $Z_r$ of singularities of index two and 
using them, they construct from $Z$ a 
$K3$ surface $X$ with a non-symplectic involution of elliptic type.
Here a {\em{$K3$ surface}} $X$ is a smooth projective surface with 
$K_X\sim 0$ and $H^1(X, \mathcal{O}_X)=0$.
An automorphism of $X$ is {\em{non-symplectic}} if it acts on $H^{2,0} (X)$ nontrivially.
An involution $\varphi$ of $X$ is 
{\em{of elliptic type}}
if the fixed locus $X^{\varphi}=\{x\in X\mid \varphi (x)=x\}$ 
contains a smooth curve of genus $g\geq 2$. 
Conversely, from a pair $(X,\varphi )$ consisting of a $K3$ surface and a non-symplectic 
involution of elliptic type, they construct
a log del Pezzo surface $Z$ of index $k\leq 2$. 
Thus the classification of $Z$ reduces to that of $(X, \varphi )$. 
Next, using a sophisticated argument of reflection groups, they define
the {\em{root invariant}} of $(X,\varphi )$ which describes the set of
negative curves (in other words, {\em{roots}}) on the right resolution $Z_r$. 
This in turn determines the set of singularities of $Z$. 
Compared to root invariants, the lattice $H^2 (X, \mathbb{Z})^{\varphi}$ is called
the {\em{main invariant}}.
This theory results in a large table of 
main invariants and (extremal) root invariants.
It contains not only the information of $\mathrm{Sing}(Z)$ but
also of the set of negative curves on the right resolution $Z_r$.

To generalize the result of Alexeev and Nikulin \cite{AN}, we define analogues of non-symplectic involutions of elliptic type 
as follows. We call a non-symplectic automorphism $\varphi$ (of any order) of a $K3$ 
surface $X$ {\em{of elliptic type}} if $X^{\varphi}$ contains a curve of genus $g\geq 2$. 
In Proposition \ref{howto} we show 
that we can construct a log del Pezzo surface
of index $k=3$ (or $1$) 
from a non-symplectic automorphism $\varphi$ of order three and of elliptic type.
Also
we obtain a necessary condition for a log del Pezzo surface $Z$ to arise from 
the pair $(X, \varphi )$:
\begin{itemize}
\item[$\star$] The linear system $|-3K_Z|$ contains a divisor of the form $2C$, where $C$ is a 
smooth curve which does not meet the singularities.
\end{itemize}
We call this property the {\em{multiple smooth divisor property}}. See the sentence 
after Proposition \ref{howto}.
It is an analogue of the smooth divisor theorem in the case of involutions.
We should notice that the multiple smooth divisor property does not hold in general and 
there are many log del Pezzo surfaces of index three which do not correspond to $K3$ surfaces,
see Section \ref{examples}.
In what follows we restrict ourselves to those $Z$ satisfying this property $\star$
to pursue the application of the theory of $K3$ surfaces.
Our Theorem \ref{cover} will show that conversely 
this condition is sufficient for $Z$ to come from a $K3$ 
surface and a non-symplectic automorphism of order three of elliptic type.

The difficulty lies in our next step, 
obtaining the list of singularities $\mathrm{Sing} (Z)$ 
because in our case $X/\varphi$ has singularities in general and 
the method of reflection group seems to be difficult to generalize.
Instead here,
we make use of the existence of $\varphi$-stable elliptic fibrations to 
overcome the problem.
Our starting point is the general Lemma \ref{enef} together with 
the topological classification of non-symplectic automorphisms of order three
on $K3$ surfaces by Artebani, Sarti \cite{AS} and the second author \cite{Taki}. 
We remark that their description of elliptic fibrations on $X$
depends on the genericity assumption of N\'{e}ron-Severi lattice, $S_X=S_X^{\varphi}$. 
There are examples of $(X, \varphi )$ which does not satisfy this condition,
see Examples \ref{exa1}, \ref{exa2}. 
To extend their description of elliptic fibrations to {\em{all}} $(X, \varphi )$ of 
elliptic type, we use
Proposition \ref{fiberwise} that describes the fiberwise information of $X^{\varphi}$. 
The combination of these local and global descriptions of fixed locus
leads us to the desired description of  
{\em{arbitrary}} $(X, \varphi)$ of elliptic type,
see Theorems \ref{ellip3}, \ref{ellip4}.  
After 
some preparatory Lemmas \ref{sectio}, \ref{nocurve1} and \ref{nocurve2}, 
we can obtain the final list of $\mathrm{Sing} (Z)$.

\begin{mt} 
\begin{enumerate}
\item Let $\varphi $ be a non-symplectic automorphism of order three on a $K3$ surface $X$ of elliptic type, i.e., $X^{\varphi}$ contains a curve of genus $g\geq 2$. Then there exists a canonical contraction of $X/\varphi $ onto a log del Pezzo surface $Z$.  The surface $Z$ has index three 
except when $S_X^{\varphi}\simeq U(3)$. 
\item Conversely if $Z$ is a log del Pezzo surface of index three which satisfies
the multiple smooth divisor property ($\star$ above), then $Z$ can be obtained by
way of (1).
\item The possible singularity $\mathrm{Sing} (Z)$ and the Picard number $\rho (Z)$ for
log del Pezzo surfaces of index three with multiple smooth divisor property are
as in the lists of Theorems \ref{ellip3} and \ref{singu}. In particular, 
there are $11$ deformation families, $33$ combinations of singularities and 
$8$ ones with Picard number $1$. 
\end{enumerate}
\end{mt}

In the final table, we can observe some similarity between lists of 
indices $k=2$ and $k=3$. The singularities of index 
three on $Z$ depend only on the "main invariant" $H^2 (X, \mathbb{Z})^{\varphi}$.
As to rational double points, within each main invariant there exists a maximal one 
and other possibilities are obtained as the Dynkin subdiagram of that.

In Section \ref{singularities} we explain log terminal singularities of index three.
The notation of singularities is fixed and explained in this section.
In Section \ref{ord3} we study non-symplectic automorphisms of order three of elliptic type.
We classify the singular fibers of 
$\varphi$-stable elliptic fibrations in Theorems \ref{ellip3} and \ref{ellip4}. 
In Section \ref{quot} we discuss the construction of log del Pezzo surfaces from $(X, \varphi )$
and obtain the list of $\mathrm{Sing} (Z)$, Theorems \ref{ellip3} and \ref{singu}.
In Section \ref{gyaku} we show conversely that the multiple smooth divisor 
property implies the existence of $(X, \varphi )$ for index $k=3$. 
In Section \ref{examples} we discuss some examples.\\

\noindent {\bf{Notation and Conventions.}} 
The symbols $A_l, D_l, E_l$ are used to denote negative-definite even lattices defined by the 
Dynkin diagrams of each type. The same symbols also denote Du Val singularities 
on log del Pezzo surfaces; from the context, it will be clear which object they actually denote.
The notation $A_l(\alpha, \beta), D_l(\alpha)$ will be used to denote singularities of index three.
Their definitions are in Section \ref{singularities}.
Here the index of a singularity is the least positive integer $k$ such that $kK_{Z}$ is Cartier 
locally around the singular point.

A hyperbolic lattice of rank $r$ is a lattice with signature $(1,r-1)$. 
We use $U$ to denote the even unimodular hyperbolic lattice of rank $2$.
For a lattice $L$, $L(n)$ is the lattice whose bilinear form is multiplied with $n$. 
A nondegenerate lattice $L$ is $p$-elementary if for the natural inclusion $L\subset L^*$
$L^*/L$ is a $p$-elementary abelian group.
 
We denote by $S_X$ the N\'{e}ron-Severi lattice of 
the $K3$ surface $X$.
Since the algebraic equivalence and linear equivalence coincide on $X$,
we often denote the equality in $S_X$ by $\sim$. We use $=$ to emphasize the 
{\em{equality of divisors}}. We say that $\varphi$ {\em{preserves}} or {\em{stabilizes}} 
a curve $C$ if $\varphi (C)=C$.
We say $C$ is {\em{fixed}} if moreover $\varphi |_C = id_C$, namely $C\subset X^{\varphi}$. 
An elliptic fibration on $X$ is $\varphi$-{\em{stable}} if for the general fiber $F$ we have 
$\varphi (F)\sim F$.

On a log del Pezzo surface $Z$, we use 
$\equiv$ to denote the numerical equivalence.

\begin{ack}
We thank Doctor Takuzo Okada for finding examples of log del Pezzo surfaces and 
helpful discussions.
We are grateful to Professor Noboru Nakayama for reading manuscript, giving
better proofs and 
suggestions. We are grateful to Professors Viacheslav V. Nikulin and JongHae Keum 
for warm encouragement and discussions. 

The first author was supported by global COE program of Kyoto University and
JSPS Grant-in-Aid (S), No. 22224001. 
His work was supported by KAKENHI 21840031. 
The second author's research was supported by Basic Science Research Program through the 
National Research Foundation (NRF) of Korea funded by the Ministry of education,
Science and Technology (2007-C00002).  
\end{ack}

\section{log terminal singularities of index three}\label{singularities}

In this section we explain log terminal singularities of index three
and fix the notation.

Two-dimensional log terminal singularities are 
quotient singularities \cite{kawamata} and 
they are classified in \cite{brieskorn}.
Recall that a subgroup $G$ of $\mathrm{GL} (2, \mathbb{C})$
is {\em{small}} if it does not contain any reflections.
Let $G$ be a finite small subgroup of 
$\mathrm{GL} (2, \mathbb{C})$
and $\mathbb{C}^2/G$ be the 
quotient singularity. Then it has index three if and only if $[G \colon G\cap \mathrm{SL} (2,
\mathbb{C})]=3$. It is not difficult to choose $G$ with this property out of 
\cite[Satz 2.9 and 2.11]{brieskorn} as follows. 
\begin{longtable}{lll} \toprule
$G$ & $\Gamma$ & conditions \\ \midrule
$C_{n,q}$ & $\langle n,q \rangle$ & $0<q<n, (n,q)=1$\\
& & and $\frac{1+q}{n}\in \frac{1}{3}\mathbb{Z}\setminus\mathbb{Z}$ \\ 
$G_0:=(\mathbb{Z}_6, \mathbb{Z}_6; \mathbb{D}_2, \mathbb{D}_2)$ &
$\langle 3;2,1;2,1;2,1 \rangle$ & \\
$G_n:=(\mathbb{Z}_6, \mathbb{Z}_6; \mathbb{D}_n, \mathbb{D}_n)$ & 
$\langle 2;2,1;2,1;n,n-3 \rangle$ & $n \not\equiv 0\ (3)$, $n\geqq 4$ \\ \bottomrule
\end{longtable}
\noindent Here $G$ is the small subgroup of $\mathrm{GL} (2, \mathbb{C})$ and 
$\Gamma$ implicitly describes the exceptional curves of minimal resolution of $\mathbb{C}^2/G$.
The group $C_{n,q}$ is a cyclic group of order $n$
generated by the matrix $\begin{pmatrix} \zeta_n  & 0 \\ 0 & \zeta_n^{q} \end{pmatrix}$
where $\zeta_n$ is a primitive $n$-th root of unity. In this case the singularity 
$\mathbb{C}^2/C_{n,q}$ is the famous Hirzebruch-Jung singularity \cite[III-5.]{BHPV}.
As such, the notation $\langle n,q \rangle$ for exceptional curves describes the chain 
of smooth rational curves whose self-intersection numbers $-b_1, \cdots, -b_r, b_i\geq 2$ 
are determined uniquely by
expressing $n/q$ as the following continued fraction: 
\begin{equation*}
b_1 - \cfrac{1}{b_2 -
	  \cfrac{1}{b_3 - 
	  \cfrac{1}{b_4 - \cdots}}}.
\end{equation*}
The notation $(\mathbb{Z}_6, \mathbb{Z}_6; \mathbb{D}_n, \mathbb{D}_n)$ of other groups in our cases 
refers just the group $G$ generated by 
\[
\begin{pmatrix} \zeta_6  & 0 \\ 0 & \zeta_6 \end{pmatrix},
\begin{pmatrix} 0  & 1 \\ -1 & 0 \end{pmatrix} \text{ and }
\begin{pmatrix} \zeta_{2n}  & 0 \\ 0 & \zeta_{2n}^{-1} \end{pmatrix}.
\]
(Among these generators the first generates the cyclic group $\mathbb{Z}_6$ and the latter two 
generate the binary dihedral group 
$\mathbb{D}_n$ of order $4n$.) As the exceptional curves of quotient singularities by these groups,
the notation $\Gamma = \langle b; n_1, q_1; n_2, q_2; n_3, q_3 \rangle $ 
describes the tree of smooth rational curves with exactly one fork with three branches, 
whose fork has self-intersection $-b$ and the three branches have the same graphs as 
$\langle n_i,q_i \rangle \ (i=1,2,3)$. See also Table \ref{CLTS3}.
\erase{Note that the former entry (resp. latter two entries) in this table come 
exactly from the first (resp. second) entry in \cite[Satz 2.11]{brieskorn}.}

In each case, by using the recursive relation of continued fractions
\[ \frac{n}{n-3}= 2-\frac{1}{\frac{n-3}{n-6}},\]
we can describe $\Gamma$ more explicitly. Given $\Gamma$, we can compute 
the discrepancies of exceptional curves easily.
These are the contents of the next table, where $l$ denotes the number of 
vertices of $\Gamma$. 
\begin{longtable}{|c|c|c|}
\hline symbol & $\Gamma$ & $G$ \\
\hline
$A_{1}$(1) & \xygraph{*+[o][F-]{-3}([]!{+(0,-.5)} {-\frac{1}{3}})} & $C_{3,1}$\\
$A_{1}$(2) & \xygraph{*+[o][F-]{-6}([]!{+(0,-.5)} {-\frac{2}{3}})} & $C_{6,1}$\\
\hline
$A_{2}$(1,2) & 
\xygraph{
*+[o][F-]{-2}([]!{+(0,-.5)} {-\frac{1}{3}}) - [r]
*+[o][F-]{-5}([]!{+(0,-.5)} {-\frac{2}{3}})
} & $C_{9,5}$
\\
$A_{2}$(2,2) & 
\xygraph{
*+[o][F-]{-4}([]!{+(0,-.5)} {-\frac{2}{3}}) - [r]
*+[o][F-]{-4}([]!{+(0,-.5)} {-\frac{2}{3}})
} & $C_{15,4}$
\\
\hline
$A_{3}$(1,1)  & 
\xygraph{
*+[o][F-]{-2}([]!{+(0,-.5)} {-\frac{1}{3}}) - [r]
*+[o][F-]{-4}([]!{+(0,-.5)} {-\frac{2}{3}}) - [r]
*+[o][F-]{-2}([]!{+(0,-.5)} {-\frac{1}{3}})
} & $C_{12,7}$
\\
$A_{3}$(1,2)  & 
\xygraph{
*+[o][F-]{-2}([]!{+(0,-.5)} {-\frac{1}{3}}) - [r]
*+[o][F-]{-3}([]!{+(0,-.5)} {-\frac{2}{3}}) - [r]
*+[o][F-]{-4}([]!{+(0,-.5)} {-\frac{2}{3}})
} & $C_{18,11}$
\\
$A_{3}$(2,2)  & 
\xygraph{
*+[o][F-]{-4}([]!{+(0,-.5)} {-\frac{2}{3}}) - [r]
*+[o][F-]{-2}([]!{+(0,-.5)} {-\frac{2}{3}}) - [r]
*+[o][F-]{-4}([]!{+(0,-.5)} {-\frac{2}{3}})
} & $C_{24,7}$
\\
\hline
$A_{l}$(1,1)  & 
\xygraph{
*+[o][F-]{-2}([]!{+(0,-.5)} {-\frac{1}{3}}) - [r]
*+[o][F-]{-3}([]!{+(0,-.5)} {-\frac{2}{3}}) - [r]
*+[o][F-]{-2}([]!{+(0,-.5)} {-\frac{2}{3}}) - [r] \cdots - [r]
*+[o][F-]{-2}([]!{+(0,-.5)} {-\frac{2}{3}}) - [r] 
*+[o][F-]{-3}([]!{+(0,-.5)} {-\frac{2}{3}}) - [r]
*+[o][F-]{-2}([]!{+(0,-.5)} {-\frac{1}{3}})
} & $C_{9l-15,6l-11}$
\\
$A_{l}$(1,2)  & 
\xygraph{
*+[o][F-]{-2}([]!{+(0,-.5)} {-\frac{1}{3}}) - [r]
*+[o][F-]{-3}([]!{+(0,-.5)} {-\frac{2}{3}}) - [r]
*+[o][F-]{-2}([]!{+(0,-.5)} {-\frac{2}{3}}) - [r] \cdots - [r]
*+[o][F-]{-2}([]!{+(0,-.5)} {-\frac{2}{3}}) - [r]
*+[o][F-]{-4}([]!{+(0,-.5)} {-\frac{2}{3}})
} & $C_{9l-9,6l-7}$
\\
$A_{l}$(2,2)  & 
\xygraph{
*+[o][F-]{-4}([]!{+(0,-.5)} {-\frac{2}{3}}) - [r]
*+[o][F-]{-2}([]!{+(0,-.5)} {-\frac{2}{3}}) - [r] \cdots - [r]
*+[o][F-]{-2}([]!{+(0,-.5)} {-\frac{2}{3}}) - [r]
*+[o][F-]{-4}([]!{+(0,-.5)} {-\frac{2}{3}})
} & $C_{9l-3,3l-2}$
\\
\hline
$D_{4}$(1)  & 
\xygraph{
*+[o][F-]{-2}([]!{+(0,-.5)} {-\frac{1}{3}}) - [r]
*+[o][F-]{-3}([]!{+(0,-.5)} {-\frac{2}{3}}) (- []!{+(1,.5)} 
*+[o][F-]{-2}([]!{+(0,-.5)} {-\frac{1}{3}}),- []!{+(1,-.5)} 
*+[o][F-]{-2}([]!{+(0,-.5)} {-\frac{1}{3}})
} & $G_0$
\\
$D_{4}$(2)  & 
\xygraph{
*+[o][F-]{-4}([]!{+(0,-.5)} {-\frac{2}{3}}) - [r]
*+[o][F-]{-2}([]!{+(0,-.5)} {-\frac{2}{3}}) (- []!{+(1,.5)} 
*+[o][F-]{-2}([]!{+(0,-.5)} {-\frac{1}{3}}),- []!{+(1,-.5)} 
*+[o][F-]{-2}([]!{+(0,-.5)} {-\frac{1}{3}})
} & $G_4$
\\
\hline
$D_{l}$(1)  & 
\xygraph{
*+[o][F-]{-2}([]!{+(0,-.5)} {-\frac{1}{3}}) - [r]
*+[o][F-]{-3}([]!{+(0,-.5)} {-\frac{2}{3}}) - [r] 
*+[o][F-]{-2}([]!{+(0,-.5)} {-\frac{2}{3}}) - [r] \cdots - [r]
*+[o][F-]{-2}([]!{+(0,-.5)} {-\frac{2}{3}}) (- []!{+(1,.5)} 
*+[o][F-]{-2}([]!{+(0,-.5)} {-\frac{1}{3}}),- []!{+(1,-.5)} 
*+[o][F-]{-2}([]!{+(0,-.5)} {-\frac{1}{3}})
} & $G_{3l-10}$
\\
$D_{l}$(2)  & 
\xygraph{
*+[o][F-]{-4}([]!{+(0,-.5)} {-\frac{2}{3}}) - [r]
*+[o][F-]{-2}([]!{+(0,-.5)} {-\frac{2}{3}}) - [r] \cdots - [r]
*+[o][F-]{-2}([]!{+(0,-.5)} {-\frac{2}{3}}) (- []!{+(1,.5)} 
*+[o][F-]{-2}([]!{+(0,-.5)} {-\frac{1}{3}}),- []!{+(1,-.5)} 
*+[o][F-]{-2}([]!{+(0,-.5)} {-\frac{1}{3}})
} & $G_{3l-8}$
\\
\hline
\caption[]{log terminal singularities of index three}\label{CLTS3}
\end{longtable}

One observation is the following: within the same type of Dynkin diagram, 
the discrepancy can vary only at the end of the diagram.
This fact in mind, we introduce the following symbols for singularities.
We denote the singularity by $A_{n}(\alpha , \beta )$ (resp. $D_{n}(\alpha )$) 
if the dual graph of the exceptional curves is of type $A_n$ (resp. $D_n$) 
and the discrepancies at the ends are 
$-\alpha /3$ and $-\beta /3$ (resp. the discrepancy at the end of the 
longest branch is $-\alpha /3$). The meaning of $A_1(\alpha )$ will be obvious.

The importance of discrepancies in defining right resolutions 
will show the advantage of these new symbols. See Section \ref{gyaku}.
In general they are adequate in numerical computations of singularities.
(Moreover, we could use the viewpoint of discrepancies to give a self-contained 
classification argument of log terminal singularities of index three.)

\erase{
\begin{rem}\label{lt-quot}
Using the basic algorithm of computing the exceptional cycles of cyclic 
quotient singularities \cite[III, Section 5]{BHPV}, we can easily deduce the 
following equalities of singularities.
\begin{equation*}
\begin{split}
A_1(1) &= \frac{1}{3}(1,1), \qquad A_1(2) = \frac{1}{6}(1,1),\\
A_2(12) &=  \frac{1}{9}(1,2), \qquad A_2(22) =  \frac{1}{15}(1,4),\\
A_3(11) &=  \frac{1}{12}(1,7), \quad  A_3(12) =  \frac{1}{18}(1,5),\\
A_3(22) &=  \frac{1}{24}(1,7), \quad  \text{ and so on.}
\end{split}
\end{equation*}
\end{rem}
}

\section{Non-symplectic automorphisms of order three on $K3$ surfaces}\label{ord3}

Let $X$ be a $K3$ surface and $\varphi$ a non-symplectic automorphism of order three
on $X$. 
In this section we extend some of the results of \cite{AS} and \cite{Taki} 
so that we will describe 
{\em{arbitrary}} non-symplectic automorphisms of order three {\em{of elliptic type}}.
Here we follow \cite{AN} to say that $\varphi$ is {\em{of 
elliptic type}} if the fixed locus $X^{\varphi}$ contains a curve $C^{(g)}$ of 
genus $g\geq 2$.
Although our main interest is in the case $\varphi$ is of elliptic type, 
let us begin with generalities. 

Let $\varphi$ be a non-symplectic automorphism of order three. 
By \cite{Nikulin}, it has no eigenvalue $1$ in the action on the transcendental lattice $T_X$. 
Therefore, the fixed sublattice $H^2 (X, \mathbb{Z})^{\varphi}$ sits inside the 
N\'{e}ron-Severi lattice $S_X$.
We denote it by $S_X^{\varphi}$. 
This lattice is easily seen to be $3$-elementary and hyperbolic, hence 
$S_X^{\varphi}$ is 
one of the $24$ lattices of \cite[Lemma 2.3]{Taki}. 
In particular $S_X^{\varphi}$ represents zero in every case.
The next lemma is crucial.

\begin{lem}\label{nef}
Let $X$ be a projective $K3$ surface, $\varphi$ be an automorphism of $X$ of order $n\geq 2$.
Assume $S_X^{\varphi}$ represents zero, i.e., there exists a divisor $D$ such that
\[ 0\not\sim D \in S_X^{\varphi},\quad (D^2)=0.\]
Then we can find another $D'$ such that 
\[ 0\not\sim D' \in S_X^{\varphi}, ((D')^2)=0 \text{ and $D'$ is nef}.\]
\end{lem}
\begin{proof}
First by the Riemann-Roch theorem we can assume that $D$ is effective.
We fix an ample divisor $H$ on $X$ and use the descending induction on the value 
$(H, D)$, noting that it is a positive integer. In this setting our 
lemma is reduced to the following claim.
\begin{equation*}
\begin{split}
\ & \textit{If $D$ is not nef, then we can find $D'$ such that } \\
\ & \textit{$0\not\sim D' \in S_X^{\varphi}$, $((D')^2)=0$, $D'$ is effective and $(H, D) > (H, D')$.}
\end{split}
\end{equation*}
\erase{
We use the following lemma. 
\begin{lem}\label{cs}\cite[IV, (7.2)]{BHPV}
Let $E$ and $F$ be effective nonzero divisors on a smooth projective surface 
with $(E^2)\geq 0, (F^2)\geq 0$. Then we have $(E,F)\geq 0$ and the equality holds 
only if the numerical classes of $E$ and $F$ are proportional. 
\end{lem}
}
Since $D$ is effective, we can make use of Zariski decomposition $D=P+N$.
Here $P$ is a nef $\mathbb{Q}$-divisor, $N$ is an effective $\mathbb{Q}$-divisor
whose prime components $\{ N_i \}$ of $N$ have negative-definite intersection matrix and 
$(P,N)=0$. By uniqueness, the negative part $N$ is determined by the numerical class of $D$, hence
$\varphi (D)\sim D$ implies $\varphi (N)=N$. 
We also note that every prime component of $N$ is a $(-2)$-curve on $X$.

Suppose $D$ is not nef. Then for  
some $(-2)$-curve $l\in \{ N_i \}$ we have $(D,l)<0$.
Let $m$ be the least positive integer such that $\varphi^m (l)=l$ and let $E= l+ \cdots
+ \varphi^{m-1} (l)$. The negativity of $N$ implies $(E^2)<0$.
From this inequality, we can classify the 
possible configuration of the divisor $E$ as follows. \\
{\bf{(I)}} When $m=1$ then $E=l$. 

If $m\geq 2$ we put $k:= (l, \varphi (l)+ \cdots + \varphi^{m-1} (l)).$ Then since
\[0>(E^2)= \sum_{i=0}^{m-1} 
\bigl( (\varphi^i (l), E-\varphi^i (l))+(\varphi^i (l), \varphi^i (l)) \bigr) =
m(k-2),\]
we obtain $k=0$ or $k=1$. \\
{\bf{(II)}} If $k=0$ then $E$ is a disjoint union of $m$ $(-2)$-curves. \\
{\bf{(III)}} If $k=1$ then the dual graph of $E$ is a disjoint union of the 
Dynkin diagram of type $A_2$. Consequently $m$ is even.\\

For a $(-2)$-element $f$ in $S_X$ we denote by $s_f$ the Picard-Lefschetz reflection
\[s_f \colon x\mapsto x+(x, f) f\]
which is an isometry of $S_X$ that preserves the positive cone.  
In case (I) we put $D'\sim s_l (D)=D+(D,l)l.$ Then $D'$ is an effective divisor (up to linear 
equivalence)
whose degree as to $H$ is less than that of $D$. It is easy to see that $D'\in S_X^{\varphi}$ 
by using $\varphi (l)=l$. Thus the claim holds. Similarly in case (II) we put 
\begin{equation*}
\begin{split}
D' &\sim s_l s_{\varphi (l)} \cdots s_{\varphi^{m-1} (l)} (D)\\
	&= D+(D,l) (l+\varphi (l)+\cdots + \varphi^{m-1} (l)).
\end{split}
\end{equation*}
In case (III) we relabel $\{ l, \cdots, \varphi^{m-1} (l)\}= \{a_1, b_1, \cdots, a_{m/2}, b_{m/2}\}$ 
so that $(a_i, b_i)=1$ holds for all $i$. Then we put 
\begin{equation*}
\begin{split}
D' &\sim s_{a_1+b_1} \cdots s_{a_{m/2}+b_{m/2}} (D)\\
	&= D+2(D,l) (a_1+b_1+\cdots + a_{m/2}+b_{m/2}).
\end{split}
\end{equation*}
In this way we obtain the claim and hence Lemma \ref{nef}.
\end{proof}
\begin{cor}\label{enef}
If $\varphi$ is a non-symplectic automorphism of order three of a $K3$ surface $X$, 
then there exists a nef divisor $D\not\sim 0$ such that $(D^2)=0$ and $\varphi (D)\sim D.$ 
In particular, 
there is a $\varphi$-stable elliptic pencil $f\colon X\rightarrow \mathbb{P}^1$.  
\end{cor}
\cite{AS, Taki} also use elliptic fibrations to describe generic $X$. To clarify 
the difference from them we note that 
under the genericity assumption $S_X^{\varphi}=S_X$, 
$\varphi$ preserves every $(-2)$-curve on $X$. 
But there are many examples where this is not the case.
\begin{ex}\label{exa1}
Let $E_0$ be the elliptic curve with period $\zeta_3=e^{2\pi i/3}$, the 
cubic root of unity, and let $\varphi$ be the 
automorphism of $E_0$ of order three given by multiplication of $\zeta_3$.
Let $E$ be another elliptic curve not
isogenous to $E_0$. We put $X= Km (E_0\times E)$, namely the minimal desingularization of the 
quotient surface $E_0\times E/ (-1)$. From the theory of Kummer surfaces \cite{BHPV}, 
we see that $X$ is a $K3$ surface of Picard number $\rho =18$; its  
transcendental lattice is isomorphic to $U(2)\oplus U(2)$. 

The automorphism $(\varphi ,1)$ of $E_0\times E$ clearly commutes with $(-1)$, hence it 
descends to an automorphism of $E_0\times E/ (-1)$.
Its fixed loci consist of
the rational curve $\{0_{E_0}\}\times (E/(-1))$ and
the elliptic curve $\{\alpha \}\times E$, where we put 
$E_0^{\varphi}=\{0_{E_0}, \alpha, -\alpha\}$. See the picture below.
In the picture, thick lines are fixed loci, circles represent singularities of $E_0\times E/ (-1)$
and the $2$-torsion points of $E_0$ (resp. $E$) are $\{0_{E_0}, a_1, a_2, a_3\}$ 
(resp. $\{0_{E}, b_1, b_2, b_3\}$).
\begin{center}
\begin{picture}(120,100)
\put(0,20){\line(1,0){120}}
\put(0,40){\line(1,0){120}}
\put(0,60){\line(1,0){120}}
\put(0,80){\line(1,0){120}}
\put(40,0){\line(0,1){100}}
\put(60,0){\line(0,1){100}} 
\put(80,0){\line(0,1){100}}
\linethickness{1.6pt}
\put(20,0){\line(0,1){100}} 
\put(112,0){\line(0,1){100}}
\linethickness{0.4pt}
\put(20,20){\circle{7}}
\put(20,40){\circle{7}}
\put(20,60){\circle{7}}
\put(20,80){\circle{7}}
\put(40,20){\circle{7}}
\put(40,40){\circle{7}}
\put(40,60){\circle{7}}
\put(40,80){\circle{7}}
\put(60,20){\circle{7}}
\put(60,40){\circle{7}}
\put(60,60){\circle{7}}
\put(60,80){\circle{7}}
\put(80,20){\circle{7}}
\put(80,40){\circle{7}}
\put(80,60){\circle{7}}
\put(80,80){\circle{7}}
\put(15,-8){$0_{E_0}$}
\put(35,-8){$a_1$}
\put(55,-8){$a_2$}
\put(75,-8){$a_3$}
\put(107,-8){$\alpha$}
\put(-8,15){$0_E$}
\put(-8,35){$b_1$}
\put(-8,55){$b_2$}
\put(-8,75){$b_3$}
\end{picture}
\end{center}
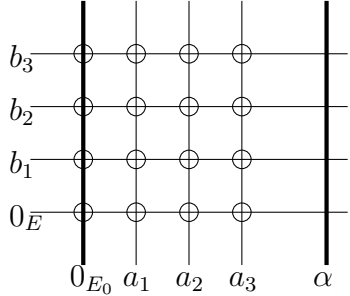
\quad \newline

The automorphism induced on $X$ is denoted by the same $\varphi$. 
On $X$ the sixteen circles are replaced by $(-2)$-curves.
Since $\varphi$ permutes circles on lines $\{a_i\}\times (E/(-1))$ in the picture, 
we see that not all $(-2)$-curves are preserved.
It is easy to 
see (see below) that the fixed points of $\varphi$ on $X$ consist of 
\[\{4 \text{ points}\} \cup \{\text{a rational curve}\}\cup \{\text{an elliptic curve}\}.\]
By \cite[Table 2]{AS} we obtain $S_X^{\varphi} \simeq U\oplus A_2^{\oplus 4}$.
Thus in fact $S_X\neq S_X^{\varphi}$. Note that by \cite{oguiso} $X$ does not have any 
Jacobian elliptic fibration with the root lattice of reducible fibers 
isomorphic to $A_2^{\oplus 4}$.

The two projections of $E_0\times E$ induce two 
$\varphi$-stable elliptic fibrations on $X$
\[f\colon X\rightarrow E/(-1)\simeq \mathbb{P}^1, \ g\colon 
X\rightarrow E_0/(-1)\simeq \mathbb{P}^1.\]
(In the picture, the fibers of $f$ are horizontal and that 
of $g$ are vertical.) 
We see that the $j$-invariant of $g$
is nonzero constant, $\varphi$ acts on the base $E_0/(-1)$
non-trivially and $g$ has the singular fiber of type $I_0^*$-(i) in the notation
of Proposition \ref{fiberwise}. 
\end{ex}
\begin{ex}\label{exa2}
Let $E_0$ be as in the previous example.
Let us consider $X=Km(E_0\times E_0)$ and its automorphism induced from $(\varphi, \varphi )$.
In this case the picture becomes 
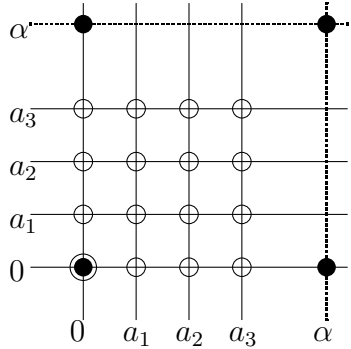
\begin{center}
\begin{picture}(120,120)
\linethickness{0.4pt}
\put(0,20){\line(1,0){120}}
\put(0,40){\line(1,0){120}}
\put(0,60){\line(1,0){120}}
\put(0,80){\line(1,0){120}}
\put(0,112){\dashbox(120,0){}}
\put(40,0){\line(0,1){120}}
\put(60,0){\line(0,1){120}} 
\put(80,0){\line(0,1){120}}
\put(20,0){\line(0,1){120}} 
\put(112,0){\dashbox(0,120){}}
\put(20,20){\circle*{7}}
\put(20,20){\circle{10}}
\put(20,40){\circle{7}}
\put(20,60){\circle{7}}
\put(20,80){\circle{7}}
\put(40,20){\circle{7}}
\put(40,40){\circle{7}}
\put(40,60){\circle{7}}
\put(40,80){\circle{7}}
\put(60,20){\circle{7}}
\put(60,40){\circle{7}}
\put(60,60){\circle{7}}
\put(60,80){\circle{7}}
\put(80,20){\circle{7}}
\put(80,40){\circle{7}}
\put(80,60){\circle{7}}
\put(80,80){\circle{7}}
\put(112,20){\circle*{7}}
\put(112,112){\circle*{7}}
\put(20,112){\circle*{7}}
\put(15,-8){$0$}
\put(35,-8){$a_1$}
\put(55,-8){$a_2$}
\put(75,-8){$a_3$}
\put(107,-8){$\alpha$}
\put(-8,15){$0$}
\put(-8,35){$a_1$}
\put(-8,55){$a_2$}
\put(-8,75){$a_3$}
\put(-8,107){$\alpha$}
\end{picture}
\end{center}
\quad \newline
Here the dashed lines are elliptic curves. 
The sixteen circles are singularities and four black circles are fixed points.
The leftmost bottom one has both properties and it gives a fixed $(-2)$-curve.
The rightmost upside black circle indicates two fixed points. 
Hence the fixed points on $X$ consist of 
\[\{4 \text{ points}\} \cup \{\text{a rational curve}\}\]
and $S_X^{\varphi}\simeq U(3)\oplus A_2^{\oplus 4}$, which is a proper sublattice 
of $S_X$ which has rank $20$.
The elliptic fibration induced from the projection has the singular fiber 
of type $I_0^*$-(ii) in the notation of Proposition \ref{fiberwise}.
It acts on the base non-trivially.
\end{ex}
 
In the rest of this section we restrict ourselves to the case of elliptic type and 
give a description of arbitrary $\varphi$.
\begin{prop}\label{ellip1}
Assumptions and $f\colon X\rightarrow \mathbb{P}^1$ as in Corollary \ref{enef}.
Assume $\varphi$ is of elliptic type. 
Then the following holds.
\begin{enumerate}
\item The $j$-invariant of $f$ is identically zero. Hence every smooth fiber is 
isomorphic to the elliptic curve with the period $\zeta_3$. 
\item The automorphism $\varphi$ acts on the base trivially.
\item The map $X^{\varphi}\rightarrow \mathbb{P}^1$ is generically $3$ to $1$. 
\item Every fiber of $f$ is of type either $I_0, I_0^*, II, II^*, IV$ or $IV^*$. 
\end{enumerate}
\end{prop} 
\begin{proof}
Since the fixed curve $C^{(g)}$ cannot be located inside a fiber, it intersects with every fiber. 
Hence (2) follows. Also every smooth fiber has an 
automorphism of order three with fixed points, hence (1) and (3).  
By the classification of singular fibers \cite[p.210]{BHPV},
if the $j$-invariant of $f$ is identically zero then 
singular fibers of $f$ are of type either $I_0, I_0^*, II, II^*, IV$ or $IV^*$.
Thus (4) follows.
\end{proof}
Thus $\varphi$ acts on each fiber $F$, $\varphi (F)=F$. 
In the next subsection we study the fixed points of this fiberwise action.
 
\subsection{Fixed points on Kodaira fibers}

We keep the elliptic fibration $f$ of Proposition \ref{ellip1}.
We can normalize the action of $\varphi$ on $H^{2,0} (X)$ as multiplication by $\zeta_3$,
the primitive cubic root of unity,
without loss of generality.
Then recall \cite{Nikulin, Taki} that 
the local action of $\varphi$ on $X$ around the fixed point is either 
\begin{equation}\label{local}
A= \begin{pmatrix} \zeta_3^2  & 0 \\ 0  & \zeta_3^2 \end{pmatrix}
\text{ (isolated) or }
B= \begin{pmatrix} 1 & 0 \\ 0 & \zeta_3 \end{pmatrix}
\text{ (fixed curve).}
\end{equation}
In the following we focus on the relationship between action {\em{on $X$}} and {\em{on $F$}}. 
We say that a fixed point $P$ of the action of $\varphi$ on $F$ is 
\begin{itemize}
\item isolated, if the local action on $X$ is given by the matrix $A$,
\item on a fixed curve, if the local action is given by $B$ and the fixed curve is inside $F$,
\item intermediate, if the local action is given by $B$ and the fixed curve is outside $F$ and 
intersects $F$ at $P$. 
\end{itemize}
The point is that we can distinguish these three types by considering only the action on $F$.
We are going to prove the following.
\begin{prop}\label{fiberwise}
Let $F$ be a singular fiber of $f$, see Proposition \ref{ellip1}.
Then the action of $\varphi$ on $F$ and its fixed points are as in one of the following pictures.
Here thick line represents a fixed curve, $\circ$ an isolated fixed point and 
$\bullet$ an intermediate fixed point on $F$. 
\begin{center}
\begin{picture}(330,140)
\put(5,130){Type $I^*_0$}
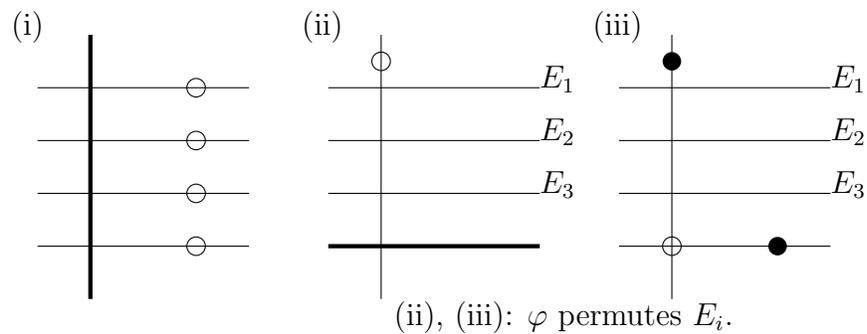
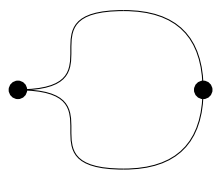
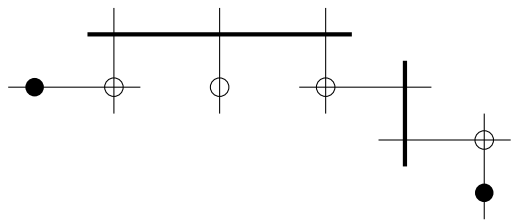
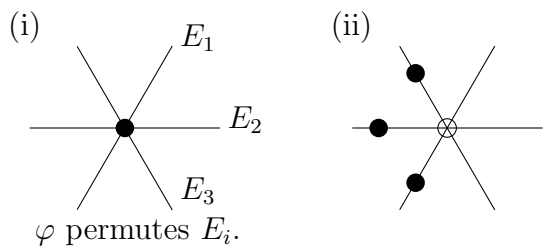
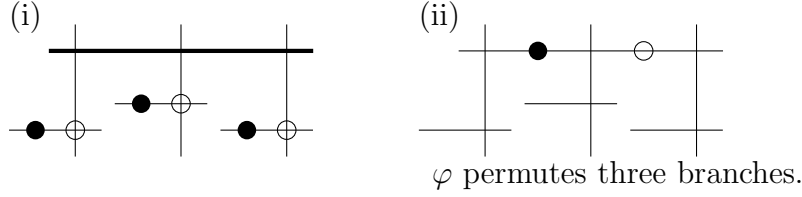
\put(5,110){(i)}
\linethickness{1.6pt}
\put(35,10){\line(0,1){100}}
\linethickness{0.4pt}
\put(15,30){\line(1,0){80}}
\put(15,50){\line(1,0){80}}
\put(15,70){\line(1,0){80}}
\put(15,90){\line(1,0){80}}
\put(75,30){\circle{7}}
\put(75,50){\circle{7}}
\put(75,70){\circle{7}}
\put(75,90){\circle{7}}
\put(115,110){(ii)}
\put(145,10){\line(0,1){100}}
\linethickness{1.6pt}
\put(125,30){\line(1,0){80}}
\linethickness{0.4pt}
\put(125,50){\line(1,0){80}}
\put(125,70){\line(1,0){80}}
\put(125,90){\line(1,0){80}}
\put(145,100){\circle{7}}
\put(205,90){$E_1$}
\put(205,70){$E_2$}
\put(205,50){$E_3$}
\put(225,110){(iii)}
\put(255,10){\line(0,1){100}}
\put(235,30){\line(1,0){80}}
\put(235,50){\line(1,0){80}}
\put(235,70){\line(1,0){80}}
\put(235,90){\line(1,0){80}}
\put(255,30){\circle{7}}
\put(255,100){\circle*{7}}
\put(295,30){\circle*{7}}
\put(315,90){$E_1$}
\put(315,70){$E_2$}
\put(315,50){$E_3$}
\put(150,0){(ii), (iii): $\varphi$ permutes $E_i$.}
\end{picture}
\end{center}
\quad \\

\begin{center}
\begin{picture}(330,100)
\put(5,100){Type $II$}
\qbezier(160,80)(190,80)(190,50)
\qbezier(190,50)(190,20)(160,20)
\cbezier(160,20)(120,20)(160,50)(120,50)
\cbezier(120,50)(160,50)(120,80)(160,80)
\put(120,50){\circle*{7}}
\put(190,50){\circle*{7}}
\end{picture}
\end{center}
\begin{center}
\begin{picture}(330,110)
\put(5,100){Type $II^*$}
\put(80,60){\line(1,0){50}}
\put(120,50){\line(0,1){40}}
\linethickness{1.6pt}
\put(110,80){\line(1,0){100}}
\linethickness{0.4pt}
\put(160,50){\line(0,1){40}}
\put(200,50){\line(0,1){40}}
\put(190,60){\line(1,0){50}}
\linethickness{1.6pt}
\put(230,70){\line(0,-1){40}}
\linethickness{0.4pt}
\put(220,40){\line(1,0){50}}
\put(260,50){\line(0,-1){40}}
\put(90,60){\circle*{7}}
\put(120,60){\circle{7}}
\put(160,60){\circle{7}}
\put(200,60){\circle{7}}
\put(260,40){\circle{7}}
\put(260,20){\circle*{7}}
\end{picture}
\end{center}
\begin{center}
\begin{picture}(330,110)
\put(5,100){Type $IV$}
\put(60,85){(i)}
\put(68,50){\line(1,0){72}}
\put(86,81){\line(18,-31){36}}
\put(86,19){\line(18,31){36}}
\put(104,50){\circle*{7}}
\put(125,81){$E_1$}
\put(143,50){$E_2$}
\put(125,23){$E_3$}
\put(70,8){$\varphi$ permutes $E_i$.}
\put(182,85){(ii)}
\put(190,50){\line(1,0){72}}
\put(208,81){\line(18,-31){36}}
\put(208,19){\line(18,31){36}}
\put(226,50){\circle{7}}
\put(200,50){\circle*{7}}
\put(214,70.7){\circle*{7}}
\put(214,29.3){\circle*{7}}
\end{picture}
\end{center}
\quad \newline
\begin{center}
\begin{picture}(330,90)
\put(5,85){Type $IV^*$}
\put(30,70){(i)}
\linethickness{1.6pt}
\put(45,60){\line(1,0){100}}
\linethickness{0.4pt}
\put(30,30){\line(1,0){35}}
\put(55,20){\line(0,1){50}}
\put(70,40){\line(1,0){35}}
\put(95,20){\line(0,1){50}}
\put(110,30){\line(1,0){35}}
\put(135,20){\line(0,1){50}}
\put(40,30){\circle*{7}}
\put(80,40){\circle*{7}}
\put(120,30){\circle*{7}}
\put(55,30){\circle{7}}
\put(95,40){\circle{7}}
\put(135,30){\circle{7}}
\put(185,70){(ii)}
\put(200,60){\line(1,0){100}}
\put(185,30){\line(1,0){35}}
\put(210,20){\line(0,1){50}}
\put(225,40){\line(1,0){35}}
\put(250,20){\line(0,1){50}}
\put(265,30){\line(1,0){35}}
\put(290,20){\line(0,1){50}}
\put(230,60){\circle*{7}}
\put(270,60){\circle{7}}
\put(190,10){$\varphi$ permutes three branches.}
\end{picture}
\end{center}
\end{prop}
\begin{proof}
We begin with the following.
\begin{lem}\label{p1}
If $C\simeq \mathbb{P}^1$ is a $(-2)$-curve in $F$ which is preserved and not fixed by $\varphi$, 
then $C$ has one isolated fixed point and the other is intermediate or on a fixed curve.
\end{lem}
\begin{proof}
By the topological Lefschetz formula $C$ has two fixed points. 
Let us choose an inhomogeneous coordinate $z$ on $C$ such that two fixed points 
are $z=0$ and $z=\infty$. If $\varphi$ acts on $z$ as a scalar $a\neq 1$, then $\varphi$ acts
on the local coordinate $z^{-1}$ near $\infty$ by $a^{-1}$.  
By (\ref{local}), we see that $\{a, a^{-1}\} = \{ \zeta_3, \zeta_3^{-1}\}$ and the 
lemma follows.
\end{proof}
The result for cases $I_0^*, II^*$ and $IV^*$ follows easily from this lemma:
we first classify the symmetry of the configuration and then can determine the 
location of fixed points. The lemma helps us to determine the types of fixed points.

For the case $IV$, we easily see that either (i) three curves are permuted
or (ii) three curves are preserved. In each case the center $Q$ is a fixed point, and 
we can blow up the center to obtain the exceptional curve $E$. The automorphism $\varphi$ 
lifts up to the blow up. 
In case (i) $\varphi$ acts on $E\simeq \mathbb{P}(T_Q X)$ nontrivially. Hence  
the center is intermediate by (\ref{local}).
In case (ii), the three intersection points of $E$ and strict transforms are fixed by $\varphi$. 
Thus $E$ is fixed by $\varphi$ and by (\ref{local}) the center is isolated. 

In case of type $II$ we have to exclude the possibilities of isolated fixed points.
First we note that the whole cusp curve cannot be fixed. This is because 
$X^{\varphi}$ contains smooth curves only.
Thus there are two fixed points $P,Q$ on $F$, where $P$ is a smooth point and $Q$ is the cusp.
Assume that either $P$ or $Q$ or both are isolated fixed points.
We consider the minimal resolution $\sigma\colon Y\rightarrow X/\varphi$ of the quotient.
Let $G:=F/\varphi\subset X/\varphi$. Then it is easy to see that $G$ is a smooth 
Weil divisor on $X/\varphi$ with $(G^2)=0$.
(In fact, in an appropriate coordinate $F$ is isomorphic to 
the cuspidal cubic $\{ zy^2=x^3 \}\subset \mathbb{P}^2$ equipped with the 
order three automorphism $(x,y,z)\mapsto (\zeta_3 x, y, z)$. Here $P$ is $(0,1,0)$ and 
$Q$ is $(0,0,1)$. )

Since in either case the isolated fixed point is of type $\frac{1}{3}(1,1)$, 
the exceptional curves $E_i\ (i\leq 2)$ of $\sigma$
are $(-3)$-curves and we obtain the relation 
\[\sigma^* G - \sum_i \frac{1}{3}E_i = \overline{G},\]
where $\overline{G}$ is the strict transform. But here the self-intersection number of 
left hand side is $(G^2)-\sum (3/9)=-(1/3) \text{ or } -(2/3) \not\in \mathbb{Z}$, 
which is impossible since $Y$ is smooth 
near $\overline{G}$. Hence we see that there are no isolated fixed points 
inside cusp fiber $F$. This concludes Proposition \ref{fiberwise}.

\end{proof}
\begin{rem}\label{abc}
\begin{enumerate}
\item If $\varphi$ is of elliptic type, then $C^{(g)}$ intersects every fiber and $F$ has 
at least one intermediate fixed point. Thus $I_0^*$-(i) and $I_0^*$-(ii) do not occur.
\item If moreover $f$ admits a fixed $(-2)$-curve which 
is a section, then obviously $IV$-(i) and $IV^*$-(ii) cannot occur.
\item On the other hand, if $\varphi$ is not of elliptic type, these actions can arise.
See Examples \ref{exa1}, \ref{exa2}.
\end{enumerate}
\end{rem}

\subsection{Automorphisms of elliptic type.}

We keep the elliptic fibration $f$ of Proposition \ref{ellip1}.
In this section we describe singular fibers of $f$ and the action of $\varphi$ on fibers 
combinatorically, relying on the results of the previous subsection.
First we recall the following
\begin{prop}[\cite{AS}, \cite{Taki}]\label{order3}
There exist exactly eleven fixed lattices for non-symplectic automorphisms of elliptic type. 
The correspondence between $S_X^{\varphi}$ and the fixed locus $X^{\varphi}$ is as follows.
\end{prop}
\begin{longtable}{|c|c|c|}
\hline
No. & $S_{X}^{\varphi}$ & $X^{\varphi }$ \\
\hline
1 & $U$ & $C^{(5)}\amalg \mathbb{P}^{1}$ \\
\hline
2 & $U(3)$ & $C^{(4)}$ \\
\hline
3 & $U\oplus A_{2}$ & $C^{(4)}\amalg \mathbb{P}^{1}\amalg \{pt\}$ \\
\hline
4 & $U(3)\oplus A_{2}$ & $C^{(3)}\amalg \{pt\}$ \\
\hline
5 & $U\oplus A_{2}^{\oplus 2}$ & $C^{(3)}\amalg \mathbb{P}^{1}\amalg \{pt\} \times 2$ \\
\hline
6 & $U(3)\oplus A_{2}^{\oplus 2}$ & $C^{(2)}\amalg \{pt\} \times 2$ \\
\hline
7 & $U\oplus E_{6}$ & $C^{(3)}\amalg \mathbb{P}^{1}\times 2 \amalg \{pt\} \times 3$ \\
\hline
8 & $U\oplus A_{2}^{\oplus 3}$ & $C^{(2)}\amalg \mathbb{P}^{1}\amalg \{pt\} \times 3$ \\
\hline
9 & $U\oplus E_{8}$ & $C^{(3)}\amalg \mathbb{P}^{1}\times 3\amalg \{pt\} \times 4$ \\
\hline
10 & $U\oplus E_{6}\oplus A_{2}$ & $C^{(2)}\amalg \mathbb{P}^{1}\times 2\amalg \{pt\} \times 4$ \\
\hline
11 & $U\oplus E_{8}\oplus A_{2}$ & $C^{(2)}\amalg \mathbb{P}^{1}\times 3\amalg \{pt\} \times 5$ \\
\hline
\end{longtable}

We divide these eleven fixed lattices into {\em{Jacobian types}} and {\em{non-Jacobian types}}. 
Non-Jacobian types consists of No. 2,4,6 and Jacobian types include others.
Equivalently $S_X^{\varphi}$ in the table is of Jacobian type if there exists an embedding of 
lattices $U\subset S_X^{\varphi}$. 

\begin{prop}\label{ellip2}
If $\varphi$ is of elliptic type and of Jacobian type, then 
after a suitable rechoice of the zero-element in Corollary \ref{enef}, 
the fibration $f$ of Proposition \ref{ellip1} has a section which is fixed by $\varphi$. 
\end{prop}
\begin{proof}
If there exists a fixed $(-2)$-curve that intersects fibers, 
then together with the curve $C^{(g)}$ 
they define a generically at least $3$ to $1$ map onto the base.
By Proposition \ref{ellip1} (2),(3) this implies that the $(-2)$-curve is a fixed section.

Let us assume that every fixed $(-2)$-curve is inside fibers. We denote by $m$ (resp. $n$) the 
number of type $II^*$ (resp. type $IV^*$-(i) ) fibers of $f$  
in the notation of Proposition \ref{fiberwise}. We remark that 
they are the only fibers that may have a 
fixed $(-2)$-curve, see also Remark \ref{abc}. Then we obtain
\begin{equation*}
\begin{split}
2m+n &= \# (\text{fixed $(-2)$-curve}),\\
4m+3n & \leq \# (\text{isolated fixed points of $\varphi$}).
\end{split}
\end{equation*}
In No.s 1,3,5,7,9,11 there exist no solutions to this restriction, hence at least one 
fixed $(-2)$-curve is outside the fibers.

In No.s 8 and 10, we have solutions $(m,n)=(0,1)$ and $=(1,0)$ respectively, and
we have 
to make a rechoice of suitable $f$. This is related to the isomorphisms of 
lattices
\[U\oplus A_2^{\oplus 3} \simeq U(3)\oplus E_6\ \  (\text{resp. } U\oplus E_6\oplus A_2 \simeq U(3)\oplus E_8).\]
We see from this isomorphism that in each case 
$S_X^{\varphi}$ has two inequivalent zero-elements, $D_1$ and $D_2$, such
that $(D_1, S_X^{\varphi})=\mathbb{Z}$ and $(D_2, S_X^{\varphi})=3\mathbb{Z}$. 
By the proof of Lemma \ref{nef} each $D_i$ can be sent to a nef element with the same property
for $i=1,2$. 

Suppose in No. 8 that we have one $IV^*$-(i) fiber $F$.
Since the components of this fiber are all preserved by $\varphi$, we can consider the 
sublattice $L\subset S_X^{\varphi}$ generated by $C=C^{(2)}$ 
and the seven components of $F$.
We will give an explicit isomorphism $L\simeq U(3)\oplus E_6$. 

The intersection numbers can be read off from Proposition \ref{fiberwise}.
The fixed curve $C$ intersects each multiplicity one component $e_1, e_2, e_3$ transversally. The seven components 
other than $e_3$ constitutes a lattice $M$ isomorphic to $E_6$. 
Let $e_i^*$ denote the element in 
$M^*$ which is dual to $e_i$, $i=1,2$. 
We can check that $e_1^*+e_2^*\in M$. 
Then we obtain
\begin{equation}\label{basis}
\begin{split}
L&= \bracket{M, e_3, C}\\
 &= \bracket{M, F, C-e_1^*-e_2^*}\\
 &\simeq U(3)\oplus E_6.
\end{split}
\end{equation}
From this isomorphism, $(m,n)=(0,1)$ occurs only if $(F, S_X^{\varphi})=3\mathbb{Z}$. 
Hence if we choose $D_1$ with $(D_1, S_X^{\varphi})=\mathbb{Z}$ as the beginning element 
in Lemma \ref{nef} the solution $(m,n)=(0,1)$ does not occur. 
In this fibration at least one fixed $(-2)$-curve is outside fibers.

The proof for No. 10 is the same.
\end{proof}
We obtain the models for Jacobian type.
\begin{thm}\label{ellip3}
Let $\varphi$ be a non-symplectic automorphism of order three of elliptic type and of Jacobian
type. Then $X$ has an elliptic pencil
$f\colon X\rightarrow \mathbb{P}^1$ which is stable under $\varphi$ and has a fixed section.
The set of singular fibers $\mathrm{Sing} (f)$ is one of the following. 
\end{thm}
\begin{longtable}{cl|lc}\toprule
No. & $\mathrm{Sing} (f)$  & $\mathrm{Sing} (Z)$ & $\rho (Z)$ \\ \midrule
1 & $12\cdot II$  & $A_1(2)$ & $1$ \\ \midrule
3a & $IV$-(ii)$+10\cdot II$  & $A_2(1,2)$ & $2$ \\ 
3b  & $I_0^*$-(iii)$+9\cdot II$  & $A_2(1,2)+A_1$ & $1$ \\ \midrule
5a & $2\cdot IV$-(ii)$+8\cdot II$  & $A_3(1,1)$ & $3$ \\
5b  & $IV$-(ii)$+I_0^*$-(iii)$+7\cdot II$  & $A_3(1,1)+A_1$ & $2$ \\
5c  & $2\cdot I_0^*$-(iii)$+6\cdot II$  & $A_3(1,1)+2 A_1$ & $1$ \\ \midrule
7 & $IV^*$-(i)$+8\cdot II$  & $D_4(2)$ & $2$ \\ \midrule
8a & $3\cdot IV$-(ii)$+6\cdot II$  & $D_4(1)$ & $4$ \\
8b  & $2\cdot IV$-(ii)$+I_0^*$-(iii)$+5\cdot II$  & $D_4(1)+A_1$ & $3$ \\
8c  & $IV$-(ii)$+2\cdot I_0^*$-(iii)$+4\cdot II$  & $D_4(1)+2 A_1$ & $2$ \\
8d  & $3\cdot I_0^*$-(iii)$+3\cdot II$  & $D_4(1)+3 A_1$ & $1$ \\ \midrule
9 & $II^* + 7\cdot II$  & $D_5(2)$ & $1$ \\ \midrule
10a & $IV^*$-(i)$+IV$-(ii)$+6\cdot II$  & $D_5(1)$ & $3$ \\
10b   & $IV^*$-(i)$+I_0^*$-(iii)$+5\cdot II$  & $D_5(1)+A_1$ & $2$ \\ \midrule
11a & $II^*+ IV$-(ii)$+5\cdot II$  & $D_6(1)$ & $2$ \\
11b   & $II^* + I_0^*$-(iii)$+4\cdot II$  & $D_6(1)+A_1$ & $1$ \\ \bottomrule
\end{longtable}
We treat the two columns on the right in the next section, see Theorem \ref{singu}.
\begin{proof}
The fibration 
$f$ is the one obtained in Proposition \ref{ellip2}.
Let $m$ and $n$ be the numbers of singular fibers of type $II^*$ and $IV^*$-(i) 
respectively. As in Proposition \ref{ellip2} we obtain 
\begin{equation*}
\begin{split}
2m+n &= \# (\text{fixed $(-2)$-curves}) -1,\\
4m+3n & \leq \# (\text{isolated fixed points of $\varphi$}).
\end{split}
\end{equation*}
In every case this has a unique solution. Thus we know the number of 
singular fibers of type $II^*$ and $IV^*$-(i). 
From Proposition \ref{fiberwise} and Remark \ref{abc} we see that other fibers are 
of type either $I_0^*$-(iii), $II$ or $IV$-(ii). The restriction on the number of 
isolated fixed points and the topological Euler number of $X$ leads us to the table.
\end{proof}
We may say that "type $IV$-(ii) fibers deform into $I_0^*$-(iii) fibers 
under the $\varphi$-equivariant deformations preserving Jacobian 
elliptic fibrations". 

In non-Jacobian types, there are no fixed $(-2)$-curves, hence $C^{(g)}$ is 
a tri-section by Proposition \ref{ellip1}. 
\begin{thm}\label{ellip4}
Let $\varphi$ be a non-symplectic automorphism of order three of elliptic type and of 
non-Jacobian type.
Let $f$ be the elliptic fibration of Proposition \ref{ellip1}.
Then the singular fibers of $f$ are as in the following table.
\end{thm}
\begin{longtable}{cll}\toprule
No. & $\mathrm{Sing} (f)$ & condition  \\ \midrule
2 & $x\cdot II+ y\cdot IV$-(i) & $x+2y=12$  \\ \midrule
4a & $IV$-(ii) $+x\cdot II+ y\cdot IV$-(i) & $x+2y=10$  \\
4b & $I_0^*$-(iii) $+x\cdot II+ y\cdot IV$-(i) & $x+2y=9$  \\
4c & $IV^*$-(ii) $+x\cdot II+ y\cdot IV$-(i) & $x+2y=8$  \\ \midrule
6a & $2\cdot IV$-(ii) $+x\cdot II+ y\cdot IV$-(i) & $x+2y=8$  \\
6b & $I_0^*$-(iii)$+IV$-(ii) $+x\cdot II+ y\cdot IV$-(i) & $x+2y=7$  \\
6c & $2\cdot I_0^*$-(iii) $+x\cdot II+ y\cdot IV$-(i) & $x+2y=6$  \\
6d & $IV$-(ii) $+IV^*$-(ii) $+x\cdot II+ y\cdot IV$-(i) & $x+2y=6$  \\
6e & $IV^*$-(ii) $+I_0^*$-(iii) $+x\cdot II+ y\cdot IV$-(i) & $x+2y=5$  \\
6f & $2\cdot IV^*$-(ii) $+x\cdot II+ y\cdot IV$-(i) & $x+2y=4$  \\ \bottomrule
\end{longtable}
\begin{proof}
The proof is the same as Jacobian case, we use the restriction on fixed $(-2)$-curves,
isolated fixed points and the topological Euler number.
\end{proof}

\section{From $K3$ surfaces to log del Pezzo surfaces}\label{quot}

We use the results of last section to obtain the list of singularities of log del Pezzo
surface $Z$ of index three with the multiple smooth divisor property.
Our discussion depends on the following observation.
\begin{prop}\label{howto}
Let $X$ be a $K3$ surface and $\varphi$ a non-symplectic automorphism of finite order $n$
such that $X^{\varphi}$ contains a curve $C$ of genus $g\geq 2$. 
Then using the natural morphisms 
\begin{equation*}
\begin{split} X &\stackrel{\nu}{\rightarrow} X_0:=\mathrm{Proj} 
\oplus_{m\geq 0} H^0 (X, \mathcal{O}_X (mC)) \\
&\stackrel{\pi}{\rightarrow} Z:=\mathrm{Proj} 
\oplus_{m\geq 0} H^0 (X, \mathcal{O}_X (mC))^{\varphi}=X_0/\varphi, 
\end{split}
\end{equation*}
we get a log del Pezzo surface $Z$ whose index divides $n$.  
Moreover $Z$ satisfies the following condition:\\
$\star$ 
The linear system $|-nK_Z|$ contains a divisor of the form $(n-1)C_0$, where $C_0=\pi \nu (C)$ is a 
smooth curve which does not meet the singularities.
\end{prop}
\begin{proof}
Since $C$ is nef and big, $\nu$ is a birational morphism
which contracts every $(-2)$-curve
on $X$ disjoint from $C$ (a very special case of Basepoint-free theorem, \cite[Theorem 3.3]{km}).
We note that for any $k$, the fixed point set $X^{\varphi^k}$ consists of $C$, $(-2)$-curves 
disjoint from $C$ and some isolated fixed points.
Therefore the induced action of $\varphi^k$ on $X_0$ has 
only fixed curve $\nu (C)$ and some isolated fixed points. 
\erase{More precisely, no curve with nontrivial stabilizer 
other than $\nu (C)$}
Also note that $\nu (C)$ is disjoint from singularities.
The ramification formula of $\pi$ is therefore
\begin{equation}\label{adj}
0\sim K_{X_0} = \pi^* K_Z + (n-1)\nu (C).
\end{equation}
Since $\nu (C)$ is ample and $\pi$ is a finite morphism, $-K_Z$ is ample.
\erase{Hartshorne III, Ex. 5.7}
Since $X_0$ has only quotient singularities, so does $Z$.
\erase{Ishii-nihongo thm 5.2.9}
Thus $Z$ is a log del Pezzo surface. 
Moreover, if $\omega$ is the nowhere vanishing 
holomorphic two form on $X$, then $\varphi$ acts trivially on
$\omega^{\otimes n}$. It descends to a nowhere vanishing section of $n K_Z$ over
$Z-\mathrm{Sing}(Z)-\pi \nu (C)$ and this is the 
local generator of $\mathcal{O}_Z (nK_Z)$ around singularities. Hence $nK_Z$ is Cartier.
The last condition follows from (\ref{adj})
by applying $\pi_*$. 
\erase{because $\mathrm{Pic} (Z)$ is torsion-free.\erase{My ref note on NMJ}}
\end{proof}
\begin{def}\label{MSDP}
Let $Z$ be a log del Pezzo surface of index $3$.
We say that $Z$ satisfies the {\em{multiple smooth divisor property}} if the linear system 
$|-3K_Z|$ contains a divisor $2C$ with $C$ a smooth curve that does not meet $\mathrm{Sing} (Z)$.
\end{def}
As the previous proposition implies, this is a natural necessary condition for $Z$ 
when we consider correspondence with $K3$ surfaces.
We will show in Section \ref{gyaku} that conversely for any log del Pezzo surface $Z$ of index 
three with this multiple smooth divisor property there exists $(X,\varphi )$
for which the construction of Proposition \ref{howto} leads to $Z$. 
We note that, 
the analogous condition for index $2$ corresponds to the smooth divisor theorem of \cite{AN}.

In the following we examine singularities of $Z$ which are
obtained from $(X, \varphi )$ via 
Proposition \ref{howto}. 
We use the elliptic fibrations obtained in Theorems \ref{ellip3}, \ref{ellip4}. 

\begin{lem}\label{sectio}
Let $f$ be the elliptic fibration of Theorem \ref{ellip3} and \ref{ellip4}.
Suppose there exists a $(-2)$-curve $E$ which is disjoint from $C=C^{(g)}$. 
Then one of the following holds.
\begin{enumerate}
\item The curve $E$ is in $X^{\varphi}$.
\item The curve $E$ is a fiber component of $f$.
\item The curves $E$, $\varphi (E)$ and $\varphi^2 (E)$ are mutually disjoint and 
they are sections of $f$. 
\end{enumerate}
\end{lem}
\begin{proof}
Let us assume that 
$E\not\subset X^{\varphi}$ and $(E, F)\geq 1$, where $F$ is the general fiber.  
We note that $(C,F)=2$ or $3$ respectively for Theorems \ref{ellip3} or \ref{ellip4}, 
by Proposition \ref{ellip1}.

We first consider the case $\varphi (E) =E$. 
Since $F$ is general, $F\cap E$ is not the fixed point of $E$, hence $(F,E)\geq 3$. 
Let us derive a contradiction using the Hodge index theorem. 
The divisor $\displaystyle C+bE-\frac{(C^2)}{(C,F)}F$ is 
orthogonal to $C$ for any $b\in \mathbb{R}$, 
hence we should have the self-intersection
\[ \left( \left(C+bE-\frac{(C^{2})}{(C,F)}F \right)^{2} \right) \leq 0.\] 
This function on $b$ takes the maximum at $b=-(C^2)(E,F)/2(C,F)$ and we deduce 
\begin{equation*}
\begin{split}
0 &\geq  \left( \left( C-\frac{(C^2)(E,F)}{2(C,F)}E-\frac{(C^2)}{(C,F)}F \right)^{2} \right) \\
   &= (C^2)\left(\frac{(C^2)(E,F)^2}{2(C,F)^2}-1 \right). \\
\end{split}
\end{equation*}
This is possible only if $(C^2)=2, (E,F)=3$ and $(C,F)=3$, namely in the case No. 6.
In this case the two fixed points of $\varphi |_E$ are both isolated since $(C,E)=0$. 
But this is a contradiction to Lemma \ref{p1}.

Next we consider the case $\varphi (E) \neq E$. Again by the Hodge index theorem,
any divisor $aE+ b\varphi (E)+ c\varphi^2 (E)$ has a non-positive 
(and negative if it is effective) self-intersection number. 
From this we see that $E, \varphi (E)$ and $\varphi^2 (E)$ are disjoint.
We put $D=E+\varphi (E)+\varphi^2 (E)$. Then $(D,F)\geq 3$ and this number is divisible by $3$. 
As in the previous case, 
the function $\left( \left( \displaystyle C+bD-\frac{(C^2)}{(C,F)}F\right)^2 \right)$
on $b\in \mathbb{R}$ takes non-positive value, 
and its maximum is 
\begin{equation*}
\begin{split}
0 &\geq \left( \left( C-\frac{(C^2)(D,F)}{6(C,F)}D-\frac{(C^2)}{(C,F)}F \right)^{2} \right)\\
  &= (C^2) \left(\frac{(C^2)(D,F)^2}{6(C,F)^2} -1\right) .
\end{split}
\end{equation*}
This inequality holds only in one of the following cases:\\
$\bullet$ We have $(C,F)=2, (C^2)=2$ and $(D,F)=3$ (No. 8, 10, 11), or\\
$\bullet$ We have $(C,F)=3, (C^2)=2,4,6$ and $(D,F)=3$ (No. 2, 4, 6).

In any case $E$ is a section of $f$. 
\end{proof}
\begin{lem}\label{nocurve1}
We use the same notation and assumptions as in Lemma \ref{sectio} (3). 
In cases of Jacobian type, the curve $E$ does not exist.
\end{lem}
\begin{proof}
Our tool is an explicit 
basis of $S_X^{\varphi}$ obtained from fiber components, c.f. (\ref{basis})
in the proof of Proposition \ref{ellip2}. 
Let $S$ be the fixed $(-2)$-curve which is a section of $f$. 
Then $U_0=\bracket{F, S}$ gives a sublattice isomorphic to $U$. 

In case of $IV$-(ii) fiber, the two components $l_1, l_2$ which are disjoint from $S$
constitute a sublattice $A_{IV}$ of $S_X^{\varphi}$ isomorphic to $A_2$ and orthogonal to 
$U_0$. 
In case of $I_0^*$-(iii) fiber, we put three simple components disjoint from 
$S$ as $l, \varphi (l), \varphi^2 (l)$ and the multiple component as $m$. Then 
$m$ and $l^+ :=l+\varphi (l)+\varphi^2 (l)$ constitute a sublattice $A_{I_0^*}$
of $S_X^{\varphi}$ 
isomorphic to $A_2$ and orthogonal to $U_0$. 
Similarly in cases of $IV^*$-(i) and $II^*$, the choice of basis is clear.

Let us give a proof in detail for No. 8. The other cases are similar.
Note that in No. 8 there are three singular fibers of type $IV$-(ii) or $I_0^*$-(iii)
which corresponds to three components $A_2$ in $S_X^{\varphi}$. 
Using the basis above for each $A_2$, we obtain the 
explicit basis for $S_X^{\varphi}$ as 
\[S_X^{\varphi}=U_0\oplus A_{IV} \oplus A_{I_0^*} \oplus \cdots,\]
where the right-hand-side should be replaced suitably.

Suppose that there exists $E$ as in Lemma \ref{sectio} (3).
Then $D=E+ \varphi (E)+ \varphi^2 (E)\in S_X^{\varphi}$. This divisor $D$ satisfies the following 
intersection relations. 
\begin{equation}\label{11}
(D, F)=3,\  (D, S)=0.
\end{equation}
\begin{equation}\label{22}
\begin{cases}
(D, l_1)=(D, l_2)=0 & \text{if $E$ meets zero component of $IV$-(ii)}\\
(D, l_1)=3,\ (D, l_2)=0 & \text{if $E$ meets $l_1$}\\
(D, l_1)=0,\ (D, l_2)=3 & \text{if $E$ meets $l_2$}\\
\end{cases}
\end{equation}
\begin{equation}\label{33}
\begin{cases}
(D, m)=(D, l^+)=0 & \text{if $E$ meets zero component of $I_0^*$-(iii)}\\
(D, m)=0,\ (D, l^+)=3 & \text{if $E$ meets $l^+$}\\
\end{cases}
\end{equation}
Since $U_0\oplus A_{IV} \oplus A_{I_0^*} \oplus \cdots$ is an 
orthogonal direct sum, $D$ has an expression $D= D_{U_0}+ D_{A_{IV}}+ D_{A_{I_0^*}}+\cdots$
and each $D_{U_0}$ etc. can be computed 
separately from the relations above. In fact
the first relation (\ref{11}) shows that 
the $D_{U_0}=3S+6F$. The second relation (\ref{22}) 
shows that $D_{A_{IV}}=0$ or $-2l_1-l_2$ or $-2l_2-l_1$. The third relation (\ref{33}) 
gives that $D_{A_{I_0^*}}= 0 $ or $-3m-2l^+$. 
Now we compute $(-2l_1-l_2)^2=(-2l_2-l_1)^2=-6$, $(-3m-2l^+)^2=-6$. 
Then we see 
\begin{equation*}
\begin{split}
(D^2)&= (D_{U_0}^2) + (D_{A_{IV}}^2)+ (D_{A_{I_0^*}}^2)+\cdots \\
&= 18 +(0 \text{ or } -6)+(0 \text{ or } -6)+\cdots \\
&\geq 18-6-6-6 \\
&=0 \\
\end{split}
\end{equation*}
since there are only three $A_2$ components. But actually $D$ consists of 
three disjoint $(-2)$-curves,
hence $(D^2)=-6$. Thus we obtain a contradiction.
(In the same way, in case No. 10 we obtain $(D^2)\geq 18-12-6 = 0$
and in case No. 11 we obtain $(D^2)\geq 18 - 0-6=12$.)
\end{proof}
In non-Jacobian cases disjoint sections can exist. To proceed, it suffices to classify the 
linear equivalence class of $D=E+\varphi (E)+\varphi^2 (E)$ by the following 
obvious lemma.
\begin{lem}
Assume that the linear system of a divisor $D_1\in S_X$ contains 
an effective divisor $E_1$ which is a disjoint union of (negative definite) 
$ADE$ configurations. Then $E_1$ is the only divisor in $|D_1|$, 
namely $H^0(\mathcal{O}_X (D_1))=1$.
\end{lem}
\begin{proof}
This is because $ADE$ configurations can be contracted to normal singularities.
\end{proof}
\begin{lem}\label{nocurve2}
We use the same notation and assumptions as in Lemma \ref{sectio} (3).
In non-Jacobian cases, we have the following possibilities for $D=E+\varphi (E)+ \varphi^2 (E)$. 
The notation of divisors will be explained in the proof.
For each class of $D\in S_X^{\varphi}$, $E$ is unique (up to $\varphi$) if exists.
\end{lem}
\begin{longtable}{cll} \toprule
No. & $D\in S_X^{\varphi}$ & intersection relations \\ \midrule
2 & $C-2F$ & \\ \midrule
4a & $C'-2F$ & $E$ meets $l_1$ \\ 
   & $C'-F-2l_2-l_3$ & $E$ meets $l_2$ \\
   & $C'-F-l_2-2l_3$ & $E$ meets $l_3$ \\
4b & $C'-2F$ & $E$ meets $n$ \\
   & $C'-F-3m-2L$ & $E$ meets $L$ \\
4c & $C'-2F$ & $E$ meets $N$ \\ \midrule
6a & $C'-2F$ & $E$ meets $l_1^1$ and $l_1^2$ \\
   & $C'-F-2l_2^2-l_3^2$ & $E$ meets $l_1^1$ and $l_2^2$ \\
   & $C'-F-l_2^2-2l_3^2$ & $E$ meets $l_1^1$ and $l_3^2$ \\
   & $C'-F-2l_2^1-l_3^1$ & $E$ meets $l_2^1$ and $l_1^2$ \\
   & $C'-2l_2^1-l_3^1-2l_2^2-l_3^2$ & $E$ meets $l_2^1$ and $l_2^2$ \\
   & $C'-2l_2^1-l_3^1-l_2^2-2l_3^2$ & $E$ meets $l_2^1$ and $l_3^2$ \\
   & $C'-F-l_2^1-2l_3^1$ & $E$ meets $l_3^1$ and $l_1^2$ \\
   & $C'-l_2^1-2l_3^1-2l_2^2-l_3^2$ & $E$ meets $l_3^1$ and $l_2^2$ \\
   & $C'-l_2^1-2l_3^1-l_2^2-2l_3^2$ & $E$ meets $l_3^1$ and $l_3^2$ \\
6b & $C'-2F$ & $E$ meets $n^1$ and $l_1^2$ \\
   & $C'-F-2l_2^2-l_3^2$ & $E$ meets $n^1$ and $l_2^2$ \\
   & $C'-F-l_2^2-2l_3^2$ & $E$ meets $n^1$ and $l_3^2$ \\
   & $C'-F-3m^1-2L^1$ & $E$ meets $L^1$ and $l_1^2$ \\
   & $C'-3m^1-2L^1-2l_2^2-l_3^2$ & $E$ meets $L^1$ and $l_2^2$ \\
   & $C'-3m^1-2L^1-l_2^2-2l_3^2$ & $E$ meets $L^1$ and $l_3^2$ \\
6c & $C'-2F$ & $E$ meets $n^1$ and $n^2$ \\
   & $C'-F-3m^2-2L^2$ & $E$ meets $n^1$ and $L^2$ \\
   & $C'-F-3m^1-2L^1$ & $E$ meets $L^1$ and $n^2$ \\
   & $C'-3m^1-2L^1-3m^2-2L^2$ & $E$ meets $L^1$ and $L^2$ \\
6d & $C'-2F$ & $E$ meets $l_1^1$ and $N^2$ \\
   & $C'-F-2l_2^1-l_3^1$ & $E$ meets $l_2^1$ and $N^2$ \\
   & $C'-F-l_2^1-2l_3^1$ & $E$ meets $l_3^1$ and $N^2$ \\
6e & $C'-2F$ & $E$ meets $N^1$ and $n^2$ \\
   & $C'-F-3m^2-2L^2$ & $E$ meets $N^1$ and $L^2$ \\
6f & $C'-2F$ & $E$ meets $N^1$ and $N^2$ \\ \bottomrule
\end{longtable}
\begin{proof}
The idea is the same as Lemma \ref{nocurve1}. 
The point is the construction of explicit basis for $S_X^{\varphi}$.

No. 2: In this case obviously $S_X^{\varphi}=\bracket{C,F}$. Assume that $E$ exists.
Then $D=E+ \varphi (E)+ \varphi^2 (E)\in S_X^{\varphi}$ satisfies $(D,C)=0, (D,F)=3$. 
Hence $D\sim C-2F$. \erase{In this case $(D^2)=-6$ and there is no contradiction. 
In fact $E$ can occur from the patching of $S_X^{\varphi}$ and its orthogonal complement
in $S_X$: $E\sim (C-2F)/3+\cdots$.} The uniqueness of $E$ follows from the previous lemma.

No. 4a: Let us put the three components of $IV$-(ii) fiber as $l_1, l_2, l_3$. 
Then
\[\bracket{C':=C+l_2+l_3, F}\oplus \bracket{l_2, l_3} \simeq U(3)\oplus A_2\]
is a basis of $S_X^{\varphi}$. When 
$E$ meets $l_1$, we have relations 
\[ (D,C)=0, (D,F)=3, (D, l_2)=(D, l_3)=0.\]
Hence $D\sim C'-2F$. 
Next when $E$ meets $l_2$, the relation becomes 
\[ (D,C)=0, (D,F)=3, (D, l_2)=3, (D, l_3)=0.\]
Then $D\sim C'-F-2l_2-l_3$.
Similarly when $E$ meets $l_3$, $D\sim C'-F-l_2-2l_3$. 

No. 4b: This has a $I_0^*$-(iii) fiber. We denote by $n$ the simple component 
preserved by $\varphi$, by $l, \varphi (l), \varphi^2 (l)$ the other simple components and 
by $m$ the double component. Then 
\[\bracket{C'=C+2m+L, F}\oplus \bracket{m, L} \simeq U(3)\oplus A_2\]
is a basis, where $L=l+\varphi (l)+\varphi^2 (l)$.
When $E$ meets $n$, the relation is 
\[ (D,C)=0, (D,F)=3, (D, m)=0, (D, L)=0.\]
Thus $D\sim C'-2F$. 
When $E$ meets $L$, the relation becomes
\[ (D,C)=0, (D,F)=3, (D, m)=0, (D, L)=3.\]
Then $D\sim C'-F-3m-2L$.

No. 4c: We denote by $m$ the triple component of $IV^*$-(ii) and by $l, \varphi (l), \varphi^2 (l)$
three double components and by $n, \varphi (n), \varphi^2 (n)$ three simple components.
We put $L=l+\varphi (l)+\varphi^2 (l), N=n+\varphi (n)+\varphi^2 (n)$. 
Then 
\[\bracket{C'=C+2m+L, F}\oplus \bracket{m,L} \simeq U(3)\oplus A_2\]
is the basis of $S_X^{\varphi}$. 
The curve $E$ can possibly meet only $N$ and we have then 
\[ (D,C)=0, (D,F)=3, (D,m)=0, (D,L)=0.\]
Thus $D\sim C'-2F$.

For cases of No. 6, we avoid describing the computations in detail. 
The notation is the same as No. 4 for fiber components and we use upper indices 
$l_1^{1}, l_1^{2}, \cdots$ to
distinguish two reducible fibers corresponding to $A_{2}^{\oplus 2}$.
We can choose the basis as 
\begin{longtable}{cl}\toprule
No. & the basis of $S_X^{\varphi}$ \\ \midrule
6a & $\bracket{C'=C+l_2^1+l_3^1+l_2^2+l_3^2, F} \oplus \bracket{l_2^1, l_3^1} \oplus \bracket{l_2^2, l_3^2}$ \\
6b & $\bracket{C'=C+2m^1+L^1+l_2^2+l_3^2, F} \oplus \bracket{m^1,L^1} \oplus \bracket{l_2^2, l_3^2}$ \\
6c & $\bracket{C'=C+2m^1+L^1+2m^2+L^2,F} \oplus \bracket{m^1,L^1} \oplus \bracket{m^2,L^2}$ \\
6d & $\bracket{C'=C+l_2^1+l_3^1+2m^2+L^2,F} \oplus \bracket{l_2^1, l_3^1} \oplus \bracket{m^2,L^2}$ \\
6e & $\bracket{C'=C+2m^1+L^1+2m^2+L^2,F} \oplus \bracket{m^1,L^1} \oplus \bracket{m^2,L^2}$ \\
6f & $\bracket{C'=C+2m^1+L^1+2m^2+L^2,F} \oplus \bracket{m^1,L^1} \oplus \bracket{m^2,L^2}$ \\ \bottomrule
\end{longtable}
and the result is as in the table.
\end{proof}
\begin{thm}\label{singu}
Let $(X, \varphi )$ be a non-symplectic automorphism of order three on a $K3$ surface. 
Let $Z$ be the log del Pezzo surface constructed in Proposition \ref{howto}.
If $\varphi$ is of Jacobian type, then $\mathrm{Sing} (Z)$ and the Picard number $\rho (Z)$  
are as in the table of Theorem \ref{ellip3}.
If $\varphi$ is of non-Jacobian type, then $\mathrm{Sing} (Z)$
is one of the following possibilities, and $\rho (Z)$ can be seen as 
\[ \rho(Z) = a - \mathrm{rk} (\text{rational double points}),\]
where $a=2,4,6$ according to No. 2, 4, 6.
\end{thm}
\begin{longtable}{c|lll}\toprule
No. & $\mathrm{Sing} (Z)$ & & \\ \midrule
2 & (nonsing.) & $A_1$ & \\ \midrule
4a & $A_1(1)$ & $A_1(1)+A_1$ & \\ 
4b & $A_1(1)+A_1$ & $A_1(1)+2 A_1$ & $A_1(1)+A_2$ \\
4c & $A_1(1)+A_2$ & $A_1(1)+A_3$ & \\ \midrule
6a & $2A_1(1)+ kA_1$ & $(0\leq  k\leq  3)$ & \\
6b & $2A_1(1)+ A_1$ & $2A_1(1)+ A_1$ & \\
   & $2A_1(1)+ A_2$ & $2A_1(1)+ A_1+A_2$ & \\
6c & $2A_1(1)+ 2A_1$ & $2A_1(1)+ 3 A_1$ & $2A_1(1)+ A_1+A_2$ \\
   & $2A_1(1)+ A_1+A_3$ & $2A_1(1)+ 2A_2$ & $2A_1(1)+ A_3$ \\
6d & $2 A_1(1)+A_2$ & $2 A_1(1)+A_3$ & \\
6e & $2 A_1(1)+ A_1+A_2$ & $2 A_1(1)+A_3+A_1$ & $2 A_1(1)+A_4$ \\
6f & $2 A_1(1)+ 2 A_2$ & $2 A_1(1)+A_5$ & \\
\bottomrule
\end{longtable}
\begin{proof}
The construction of $Z$ from $X$ is given by 
\begin{itemize}
\item the contraction of all $(-2)$-curves disjoint from $C^{(g)}$, and
\item the quotient by $\varphi$.
\end{itemize}
Equivalently we can construct $Z$ as in the following diagram:
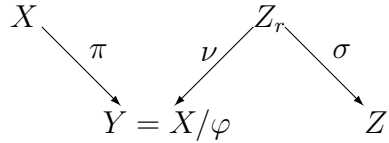
\begin{center}
\begin{picture}(300,60)
\put(77,45){$X$}
\put(89,45){\vector(1,-1){30}}
\put(112,5){$Y=X/\varphi$}
\put(169,45){\vector(-1,-1){30}}
\put(169,45){$Z_r$}
\put(181,45){\vector(1,-1){30}}
\put(211,5){$Z$}
\put(107,33){$\pi$}
\put(149,32){$\nu$}
\put(199,33){$\sigma $}
\end{picture}
\end{center}
where $\nu$ is the minimal resolution and $\sigma $ is the contraction away from 
the transform of $C^{(g)}$ ($Z_r$ is the one called right resolution of $Z$, 
see the next section).
In this diagram,
a singularity $A_1(1)$ on $Y$ appears from an isolated fixed point on $X$.
A $(-2)$-curve $E$ on $X$ which is preserved and not fixed by $\varphi$ is 
mapped to a $(-1)$-curve on $Z_r$. We can compute the diagram explicitly for each case of 
Theorems \ref{ellip3}, \ref{ellip4}.
Since we already know all the curves disjoint from $C^{(g)}$, Lemmas 
\ref{sectio}, \ref{nocurve1} and \ref{nocurve2}, the computation 
of $\mathrm{Sing} (Z)$ can be done. 

The computations for Jacobian cases are straightforward. We can draw a detailed picture using 
Theorem \ref{ellip3}, Lemmas \ref{sectio}, \ref{nocurve1}. We omit the explanation.

For non-Jacobian cases, we take up No. 4b case for example.
The other cases are done in a similar way.
In No. 4b we have a $I_0^*$-(iii) fiber. By Lemma \ref{nocurve2} we have the following 
possibilities.
\begin{enumerate}
\item There are no sections.
\item Only $D\sim C'-2F$ is realized by sections.
\item Only $D\sim C'-F-3m-2L$ is realized by sections.
\item Both $D\sim C'-2F$ and $D\sim C'-F-3m-2L$ are realized by sections.
\end{enumerate}
But (4) does not occur, because $(C'-2F, C'-F-3m-2L)=-3< 0$ hence it cannot happen that 
both are effective divisors without common components.
In (1) we have one $A_1(1)$ corresponding to the isolated fixed point and 
one $A_1$ corresponding to the three permuted simple components of $I_0^*$-(iii).
In (2) the three sections produce one more $A_1$. On the other hand in (3) the three
sections intersect the three components of $I_0^*$-(iii) and produce $A_2$ instead of $A_1$. 
\end{proof}
\begin{cor}
Except for the case $S_X^{\varphi}\simeq U(3)$, we obtain  
a log del Pezzo surface of index three by Proposition \ref{howto}.
\end{cor}
\begin{rem}\label{nontoric}
In cases other than No. 7 and 10, the log del Pezzo surface $Z$ with the 
maximal rank of rational double points has the Picard number $\rho (Z)=1$. 
In particular in No.s 8d, 9 and 11b we get $Z$ with non-cyclic quotient singularities, 
hence different from (quasi-smooth) toric examples. 
\end{rem}

\erase{\begin{longtable}{|c|c|c|}
\hline
No. & $S_{X}^{\varphi}$ & Singularities of a log del Pezzo surface \\
\hline
1 & $U$ & $A_{1}(2)$ \\
\hline
2 & $U(3)$ & non \\
\hline
3 & $U\oplus A_{2}$ & $A_{2}(12)$  \\
\hline
4 & $U(3)\oplus A_{2}$ & $A_{1}(1)$ \\
\hline
5 & $U\oplus A_{2}^{\oplus 2}$ & $A_{3}(11)$ \\
\hline
6 & $U(3)\oplus A_{2}^{\oplus 2}$ & $A_{1}(1) \times 2$  \\
\hline
7 & $U\oplus E_{6}$ & $D_{4}(2)$  \\
\hline
8 & $U\oplus A_{2}^{\oplus 3}$ & $D_{4}(1)$ \\
\hline
9 & $U\oplus E_{8}$ & $D_{5}(2)$ \\
\hline
10 & $U\oplus E_{6}\oplus A_{2}$ & $D_{5}(1)$ \\
\hline
11 & $U\oplus E_{8}\oplus A_{2}$ & $D_{6}(1)$ \\
\hline
\caption[]{N\'{e}ron-Severi lattices and singularities}\label{SandS}
\end{longtable}}

\section{From log del Pezzo surfaces to $K3$ surfaces}\label{gyaku}

In this section we prove the following theorem. 
\begin{thm}\label{cover}
Let $Z$ be a log del Pezzo surface of index three.
Assume that it satisfies the multiple smooth divisor property, namely
the linear system $|-3K_{Z}|$ contains a divisor $2C$ 
with $C$ a non-singular curve which does not meet singularities. Then
there exist a $K3$ surface $X$ and a non-symplectic automorphism $\varphi$ 
of order three of elliptic type on $X$ such that
$Z$ can be obtained from $(X,\varphi )$ by the construction of Proposition
\ref{howto}.
\end{thm}

The proof uses standard constructions and classification theory of 
surfaces. We begin with the following remark.
By \cite[III, Corollary 7.9.]{hartshorne} $C$ is connected.
Let $g$ be the genus of $C$. 
Since $C$ is located in the smooth locus, and by the conditions, 
the genus formula shows 
\begin{equation}\label{fff}
 2g-2=(C^2)+(C,K_Z)=-(C,K_Z)/2>0.
\end{equation} 
Hence $g\geq 2$. 

\subsection{Right resolution.}
Let $\tilde{\sigma }: \widetilde{Z}\rightarrow Z$ be the minimal resolution of singularities.
We denote by $E_i$ an exceptional curve over a singularity of index three and 
$a_i$ its discrepancy.
We consider the following blowings up of $\widetilde{Z}$. 
\begin{itemize}
\item If exceptional curves satisfy 
$(E_{i},E_{j})=1$ and $a_{i}=-1/3$, $a_{j}=-2/3$ then we blow up at $E_{i}\cap E_{j}$.
\item If exceptional curves satisfy $(E_{i},E_{j})=1$ and $a_{i}=a_{j}=-2/3$, then
after the blow up at $E_{i}\cap E_{j}$, we again blow up the two intersection points
of the three exceptional curves. 
\item We remark that
there exist no exceptional curves $E_{i}$, $E_{j}$ such that 
$(E_i ,E_i)=1$ and both have discrepancies $-1/3$, see Table \ref{CLTS3}. 
\end{itemize}
We do this process for all pairs $(E_i, E_j)$.
Then we obtain the surface $Z_r \stackrel{\sigma }{\rightarrow} Z$, 
which we call the \textit{right resolution} of $Z$, 
whose exceptional divisor over a singularity of index three is 
a successive union of the unit chain 
\[
\xygraph{
*+[o][F-]{-3}([]!{+(0,-.5)} {-\frac{1}{3}}) - [r]
*+[o][F-]{-1}([]!{+(0,-.5)} {0}) - [r]
*+[o][F-]{-6}([]!{+(0,-.5)} {-\frac{2}{3}})
},
\]
(or one $(-3)$-curve for $A_1(1)$ or one $(-6)$-curve for $A_1(2)$).
For example, the minimal resolution of $A_3(1,2)$ with components $E_1, E_2, E_3$ in this order
as in Table \ref{CLTS3} will be blown up to the chain 
\[
\xygraph{
*+[o][F-]{-3}([]!{+(0,-.5)} {-\frac{1}{3}})([]!{+(0,.5)} {E_1}) - [r]
*+[o][F-]{-1}([]!{+(0,-.5)} {0}) - [r]
*+[o][F-]{-6}([]!{+(0,-.5)} {-\frac{2}{3}})([]!{+(0,.5)} {E_2}) - [r]
*+[o][F-]{-1}([]!{+(0,-.5)} {0}) - [r]
*+[o][F-]{-3}([]!{+(0,-.5)} {-\frac{1}{3}}) - [r]
*+[o][F-]{-1}([]!{+(0,-.5)} {0}) - [r]
*+[o][F-]{-6}([]!{+(0,-.5)} {-\frac{2}{3}})([]!{+(0,.5)} {E_3})
}.
\]
The point is that curves with nonzero discrepancies are disjoint to each other on $Z_r$.
Let $p$ (resp. $q$) be the number of $(-3)$-curves (resp. $(-6)$-curves) in the exceptional 
locus $\mathrm{Exc} (\sigma )$ of $Z_r$. 
Now we relabel $(-3)$-curves as $E_i, 1\leq i\leq p$ and $(-6)$-curves as $F_i, 1\leq i\leq q$.
Then the comparison of canonical bundles for $f$ shows
\[K_{Z_r} \equiv \sigma ^* K_Z -\frac{1}{3}\sum_{p} E_i -\frac{2}{3}\sum_{q} F_i. \]

\subsection{Branched covering.}
Since $3K_Z$ is Cartier, we have the relation
\begin{equation}\label{yy}
\begin{split}
-3K_{Z_r}&= \sigma ^*(-3K_Z) + \sum_{p} E_i+2\sum_{q} F_i\\
&\sim 2C + \sum_{p} E_i+2\sum_{q} F_i,
\end{split}
\end{equation}
where we denoted the strict transform of $C$ on $Z_r$ by the same $C$. 
By taking the branched cover
with branch $2C + \sum E_i+2\sum F_i$ together with normalization, we get a triple cover
$\pi \colon \widetilde{X}\rightarrow Z_r$, simply branched over the disjoint union
$C\sqcup (\sqcup E_i)\sqcup (\sqcup F_i)$. 

We put $\widetilde{E}_i=\pi^* (E_i)_{\mathrm{red}}, \widetilde{F}_i=\pi^* (F_i)_{\mathrm{red}}$
and $\widetilde{C}=\pi^* (C)_{\mathrm{red}}$. They are $(-1)$-curves, $(-2)$-curves
and a curve of genus $g$ on $\widetilde{X}$ respectively.  
We have the ramification formula
\begin{equation*}
\begin{split}
3K_{\widetilde{X}}&= 3\pi^* K_{Z_r} + 6\widetilde{C} + 6\sum_{p} \widetilde{E}_i
+6 \sum_{q} \widetilde{F}_i
\end{split}
\end{equation*}
and by substituting (\ref{yy}), we get 
$3K_{\widetilde{X}}\sim 3\sum_{p} \widetilde{E}_i$. 
Since $\widetilde{E}_i$ are disjoint $(-1)$-curves, we can contract them 
and we get a surface $X$ with $3K_X \sim 0$. 

\subsection{$X$ is a $K3$ surface.}
Let us show that $X$ is a $K3$ surface. 
By $3K_X\sim 0$, 
it is a minimal surface with Kodaira dimension $\kappa =0$. 
Recall that the class of minimal algebraic surfaces with $\kappa =0$ consists 
of $K3$ surfaces, Enriques surfaces, abelian surfaces and bi-elliptic (or hyperelliptic) surfaces
\cite{BHPV}. An Enriques surface has $K_Y\not\sim 0$ and $2K_Y\sim 0$, so $X$ is 
not an Enriques surface. 
Among other three surfaces, we can distinguish $K3$ surfaces by showing 
$H^1 (X, \mathcal{O}_X)=0$. 
Clearly this is equivalent to saying $H^1 (\widetilde{X}, 
\mathcal{O}_{\widetilde{X}})=0$. We will prove this vanishing.

By \cite[Claim 3.10]{EV}, putting $B:=2C + \sum E_i+2\sum F_i$, we have 
\begin{equation*}
\begin{split}
\pi_* \mathcal{O}_{\widetilde{X}} &= \mathcal{O}_{Z_r} \oplus 
\biggl( \mathcal{O}_{Z_r}(K_{Z_r} + \lfloor \frac{1}{3}B \rfloor )\biggr) \oplus 
\biggl(\mathcal{O}_{Z_r}(2K_{Z_r} + \lfloor \frac{2}{3}B \rfloor )\biggr) \\
&= \mathcal{O}_{Z_r} \oplus \mathcal{O}_{Z_r} (K_{Z_r}) \oplus 
\mathcal{O}_{Z_r} (2K_{Z_r}+C+\sum F_i).
\end{split}
\end{equation*}
Since $Z_r$ is a rational surface \cite[Lemma 1.3]{AN}, we know $H^1$ of the first two 
components vanish. Hence we have only to show 
\[H^1 \bigl( Z_r, \mathcal{O}_{Z_r} (2K_{Z_r}+C+\sum F_i)\bigr) =0. \]
We use the exact sequence 
\[\begin{split}
0 &\rightarrow \mathcal{O}_{Z_r} (2K_{Z_r}+C+\sum F_i) \\ 
&\rightarrow
\mathcal{O}_{Z_r} (2K_{Z_r}+C+\sum E_i+\sum F_i) \rightarrow \oplus \mathcal{O}_{E_i}(-1) 
\rightarrow 0.
\end{split}\]
Taking the cohomology we obtain 
\begin{equation}\label{aaaa}
\begin{split}
0&=\oplus H^0 (\mathcal{O}_{E_i}(-1)) \rightarrow H^1 (\mathcal{O}_{Z_r} (2K_{Z_r}+C+\sum F_i))\\
&\rightarrow H^1 (\mathcal{O}_{Z_r} (2K_{Z_r}+C+\sum E_i+\sum F_i)).
\end{split}
\end{equation}
In this sequence 
the last term can be written as $H^1 (K_{Z_r}+\lceil K_{Z_r}+C+\frac{1}{3}\sum E_i+
\frac{2}{3}\sum F_i\rceil)$. Since $\frac{1}{3}C$ is nef and big, numerical equivalence
\[K_{Z_r}+C+\frac{1}{3}\sum E_i+\frac{2}{3}\sum F_i\equiv \frac{1}{3}C\] 
shows $H^1 (\mathcal{O}_{Z_r} (2K_{Z_r}+C+\sum E_i+\sum F_i))=0$ by Kawamata-Viehweg
vanishing theorem. Hence the middle term of (\ref{aaaa}) also vanishes and 
we have proved $X$ is a $K3$ surface. 

The rest is clear: the covering transformation
of $\widetilde{X}\rightarrow Z_r$ produces a non-symplectic automorphism of 
order three on $X$. It is of elliptic type because $C$ has 
genus $\geq 2$. Thus Theorem \ref{cover} is proved.

\begin{rem}
It is easy to see from (\ref{fff}) that if $Z$ has the 
multiple smooth divisor property then $K_Z^2=8(g-1)/3$. In particular, it is necessary
for $3K_Z^2$ to be divisible by $8$. This is a useful criterion.
For example, the easiest log del Pezzo surface
$Z=\mathbb{P} (1,1,3)$ has $3K_Z^2=25$ hence it does not satisfy multiple smooth divisor
property.
\end{rem}

\section{Toric examples}\label{examples}

In this section we collect examples of log del Pezzo surfaces of index three
obtained as toric varieties. 
We thank Dr. T. Okada for providing us with the computer searching program.
For the notation of singularities, we refer to Section \ref{singularities}.

\begin{ex}
We put $Z:=\mathbb{P} (1,1,6)$. 
It is easy to see  that $Z$ is a log del Pezzo surface of index three 
and has a singularity of type $A_{1}(2)$ at $(0,0,1)$.
We remark $|-3/2K_{Z}|=|\mathcal{O}_{Z}(12)|$.
Let $C$ be an element of $\mathcal{O}_{Z}(12)$ defined by 
$\{x^{12}+y^{12}+z^{2}+\text{(terms of degree 12)}=0\}$ 
where $x$, $y$ and $z$ are homogeneous coordinates of $\mathbb{P} (1,1,6)$.
Then a smooth divisor $C$ does not pass through $(0,0,1)$.
Hence $Z$ satisfies the multiple smooth divisor property.
\end{ex}

\begin{ex}
We consider the weighted hypersurface 
$Z\subset \mathbb{P} (1,1,1,3)$ of degree $4$
(We denote by $Z=(4)\subset \mathbb{P} (1,1,1,3)$ the general one). 
Note that $Z$ is a log del Pezzo surface with a singular point of type $A_{1}(1)$.
In particular the singular point is induced by $(0,0,0,1)$.
We remark $\mathcal{O}_{Z}(K_{Z})\simeq \mathcal{O}_{Z}(4-1-1-1-3)=\mathcal{O}_{Z}(-2)$.
Let $C$ be an element of $\mathcal{O}_{Z}(3)$ defined by 
$\{x^{3}+y^{3}+z^{3}+w+\cdots =0\}$
where $x$, $y$, $z$ and $w$ are homogeneous coordinates of $\mathbb{P} (1,1,1,3)$.
Then a smooth divisor $C$ does not pass through $(0,0,0,1)$.
Hence $Z$ satisfies the multiple smooth divisor property.
\end{ex}

The following table contains some examples of log del Pezzo surfaces of index three.
The first eighteen examples are in \cite[Theorem 1.3]{Dais}.
The notation $\bigcirc $ implies that 
a log del Pezzo surface $Z$ satisfies the multiple smooth divisor property.
On the other hand $\times $ is not so.

\begin{center}
\begin{longtable}{cccc}\toprule 
$Z$ & $\mathrm{Sing} (Z)$ & $K^{2}$& $2C\in |-3K_Z|$ \\\midrule
$\mathbb{P} (1,1,3)$ & $A_1(1)$ & 25/3 & $\times $ \\
$\mathbb{P} (1,3,4)$ & $A_1(1)+A_{3}$ & 16/3 & $\bigcirc $ \\
$\mathbb{P} (2,3,5)$ & $A_1(1)+A_1+A_4$ & 10/3 & $\times $ \\
$\mathbb{P} (1,1,2)/C_3$ & $2A_1(1)+A_{5}$ & 8/3 & $\bigcirc $  \\
$\mathbb{P} (1,1,6)$ & $A_1(2)$ & 32/3 & $\bigcirc $ \\ \midrule
$\mathbb{P} (1,6,7)$ & $A_{1}(2)+A_{6}$ & 14/3 & $\times $ \\
$\mathbb{P} (1,3,4)/C_2$ & $A_{1}(2)+A_{1}+A_{7}$ & 8/3 & $\times $ \\
$\mathbb{P} (1,2,3)/C_3$ & $A_{1}(1)+A_{1}(2)+A_{8}$ & 2 & $\times $ \\
$\mathbb{P}^2/ C_9$ & $2A_{2}(1,2)+A_{8}$ & 1 & $\times $ \\
$\mathbb{P} (1,5,9)$ & $A_{2}(1,2)+A_{4}$ & 5 & $\times $ \\ \midrule
$\mathbb{P} (1,2,9)$ & $A_{2}(1,2)+A_{1}$ & 8 & $\bigcirc $ \\
$\mathbb{P} (1,2,3)/C_3$ & $A_{2}(1,2)+A_{2}+A_{5}$ & 2 & $\times $ \\
$\mathbb{P} (1,1,2)/C_2\times C_3$ & $2A_{1}(2)+A_{11}$ & 4/3 & $\times $ \\
$\mathbb{P} (1,1,6)/C_2$ & $A_{3}(1,1)+2A_{1}$ & 16/3 & $\bigcirc $  \\
$\mathbb{P} (1,4,15)$ & $A_{2}(2,2)+A_{3}$ & 20/3 & $\times $ \\ \midrule
$\mathbb{P} (1,1,3)/C_5$ & $A_{2}(2,2)+2A_{4}$ & 5/3 & $\times $ \\
$\mathbb{P} (1,2,9)/C_2$ & $A_{3}(1,2)+A_1+A_{3}$ & 4 & $\times $ \\
$\mathbb{P} (1,1,6)/C_4$ & $A_{3}(2,2)+2A_{3}$ & 8/3 & $\times $ \\ \midrule 
\midrule
$(4)\subset \mathbb{P} (1,1,1,3)$ & $A_1(1)$ & 16/3 & $\bigcirc $ \\
$(7)\subset \mathbb{P} (1,1,1,6)$ & $A_1(2)$ & 14/3 &  $\times $ \\
$(5)\subset \mathbb{P} (1,1,2,3)$ & $A_1(1)+A_1$ & 10/3 & $\times $ \\
$(8)\subset \mathbb{P} (1,1,2,6)$ & $A_1(2)+A_1$ & 8/3 & $\times $ \\
$(10)\subset \mathbb{P} (1,1,2,9)$ & $A_2(1,2)$ & 5 &$\times $ \\ \midrule 
$(6)\subset \mathbb{P} (1,1,3,3)$ & $A_1(1)$ & 8/3 & $\times $ \\
$(9)\subset \mathbb{P} (1,1,3,6)$ & $A_1(1)+A_1(2)$ & 2 & $\times $ \\
$(16)\subset \mathbb{P} (1,1,4,15)$ & $A_2(2,2)$ & 20/3 & $\times $ \\
$(10)\subset \mathbb{P} (1,1,5,9)$ & $A_2(1,2)$ & 8 & $\bigcirc $ \\
$(12)\subset \mathbb{P} (1,1,6,6)$ & $A_1(2)$ & 4/3 & $\times $ \\ \midrule
$(12)\subset \mathbb{P} (1,2,3,9)$ & $A_2(1,2)+A_2$ & 2 & $\times $ \\
$(18)\subset \mathbb{P} (1,2,9,9)$ & $A_2(1,2)$ & 1 & $\times $ \\
$(14)\subset \mathbb{P} (1,2,7,12)$ & $A_3(1,1)+A_1$ & 16/3 & $\bigcirc $ \\
$(15)\subset \mathbb{P} (1,5,6,9)$ & $A_2(1,2)+A_5$ & 2 & $\times$ \\ \bottomrule
\end{longtable}
\end{center}

\begin{rem}
For example, 
we consider $\mathbb{P} (1,2,9)/C_2$.
The group $C_2=\langle g \rangle $ acts on $\mathbb{P} (1,2,9)$ via 
$g:(x,y,z)\rightarrow (x,-y,-z)$.

We note $(x,-y,-z)=(e_{4}x,y,e_{4}^{3}z)=(-x,-y,z)$ 
where $e_{k}$ is a primitive $k$-th root of unity.
The fixed locus of the action consists of three points.

At $(0,0,1)$, we have the open set $\mathbb{A}_{x,y}^{2}/\langle g, h \rangle $
where $g:(x,y)\rightarrow (-x,-y)$ and $h:(x,y)\rightarrow (e_{9}x,e_{9}^{2}y)$.
Then we have a singularity of type $1/18(1,5)$.

Similarly in other points we can calculate.
Hence we have singularities of type $A_{1}$ at $(1,0,0)$, 
of type $A_{2}$ at $(0,1,0)$ and of type $A_{3}(1,2)$ at $(0,0,1)$.
\end{rem}

\begin{rem}
There are two $\mathbb{P} (1,2,3)/C_3$ in Dais's list \cite[Theorem 1.3]{Dais}.
Indeed the group $C_{3}=\langle g_{i} \rangle $ has two actions on $\mathbb{P} (1,2,3)$, 
$g_{1}:(x,y,z)\rightarrow (x,e_{3}y,e_{3}z)$ and 
$g_{2}:(x,y,z)\rightarrow (x,e_{3}y,e_{3}^{2}z)$ 
where $e_{3}$ is a primitive third root of unity.

In the case of $C_{3} =\langle g_{1} \rangle $, we note
$(x,e_{3}y,e_{3}z)=(e_{3}x,y,e_{3}z)=(e_{9}^{2}x,e_{9}^{7}y,z)$.
Then we have singularities of type $A_{1}(1)$ at $(1,0,0)$, 
of type $A_{1}(2)$ at $(0,1,0)$ and of type $A_{8}$ at $(0,0,1)$.

In the case of $C_{3} =\langle g_{2} \rangle $, we note
$(x,e_{3}y,e_{3}^{2}z)=(e_{3}x,y,e_{3}^{2}z)=(e_{9}x,e_{9}^{5}y,z)$.
Then we have singularities of type $A_{2}$ at $(1,0,0)$, 
of type $A_{5}$ at $(0,1,0)$ and of type $A_{2}(1,2)$ at $(0,0,1)$.
\end{rem}

\begin{rem}
The group $C_{3}$ acts on $\mathbb{P}(1,1,2)$ via 
$g:(x,y,z)\rightarrow (x,e_{3}y,e_{3}z)$.

The group $C_{2}$ acts on $\mathbb{P}(1,3,4)$ via
$g:(x,y,z)\rightarrow (x,-y,-z)$.

The group $C_{9}$ acts on $\mathbb{P}^{2}$ via
$g:(x,y,z)\rightarrow (x,e_{9}y,e_{9}^{2}z)$.

The group $C_{2}\times C_{3}$ acts on $\mathbb{P}(1,1,2)$  via
$g:(x,y,z)\rightarrow (-x,e_{3}y,-e_{3}z)$.

The group $C_{2}$ acts on $\mathbb{P}(1,1,6)$  via
$g:(x,y,z)\rightarrow (x,-y,-z)$.

The group $C_{5}$ acts on $\mathbb{P}(1,1,3)$ via
$g:(x,y,z)\rightarrow (x,e_{5}y,e_{5}^{4}z)$.

The group $C_{4}$ acts on $\mathbb{P}(1,1,6)$ via 
$g:(x,y,z)\rightarrow (x,e_{4}y,e_{4}^{3}z)$.
\end{rem}

\end{document}